\documentclass[11pt]{article}
\usepackage{amssymb}
\usepackage{amsmath}
\usepackage{amsthm}
\usepackage{epsfig}
\usepackage{mathrsfs}
\usepackage{graphicx}
\usepackage{ textcomp }
\usepackage{xcolor}
\usepackage{hyperref}
\usepackage{bbold}

\def\RR{\mathbb{R}}

\def\ZZ{\mathbb{Z}}
\def\TT{\mathbb{T}}

\def\sgn{\mathrm {sgn}\,}

\def\la{\langle}
\def\ra{\rangle}

\def\to{\rightarrow}
\def\tto{\longrightarrow}

\def\vphi{\varphi}
\def\pa{\partial}
\def\na{\nabla}
\def\eps{\varepsilon}

\def\hfe{\widehat {f^\eps}}

\def\Wb{\overline W}
\def\Tb{\overline T}

\def\om{\omega}

\newtheorem{theorem}{Theorem}[section]
\newtheorem{lemma}[theorem]{Lemma}
\newtheorem{proposition}[theorem]{Proposition}

\newtheorem{remark}[theorem]{Remarks}
\newtheorem{rk&ex}[theorem]{Remarks \& Examples}
\newtheorem{corollary}[theorem]{Corollary}

\newcommand{\beqar}{\begin{eqnarray*}}
\newcommand{\eeqar}{\end{eqnarray*}}
\newcommand{\lp}{\left(}
\newcommand{\rp}{\right)}

\newcommand{\be}{\begin{equation}}
\newcommand{\ee}{\end{equation}}

\def\signam{\bigskip \begin{center} {\sc
Antoine Mellet\par\vspace{3mm}
University of Maryland\par
Department of Mathematics \par
College Park 20742\par
USA\par\vspace{3mm}
e-mail:} \tt{mellet@math.umd.edu} \end{center}}

\def\signsm{\bigskip \begin{center} {\sc
Sara Merino-Aceituno\par\vspace{3mm}
University of Cambridge\par
DPMMS, Centre for Mathematical Sciences\par
Wilberforce Road,
Cambridge CB3 0WA,
UK\par\vspace{3mm}
e-mail:} \tt{s.merino-aceituno@maths.cam.ac.uk} \end{center}}

\title{Anomalous energy transport in FPU-$\beta$ chain\footnote{The final publication is available at Springer via \texttt{ http://dx.doi.org/10.1007/s10955-015-1273-2}}}

\author{A. Mellet\thanks{University of Maryland, College Park MD, 20742, USA.} \mbox{ }and S. Merino-Aceituno\thanks{University of Cambridge, Cambridge CB3 0WA, UK}\\
}

\begin{document}

\maketitle

\begin{abstract}
This paper is devoted to the derivation of a macroscopic fractional diffusion equation describing heat  transport in an anharmonic chain. More precisely, we study here the so-called FPU-$\beta$ chain, which is a very simple model for a one-dimensional crystal in which atoms are coupled to their nearest neighbors by a harmonic potential, weakly perturbed by a quartic potential.
The  starting point of our mathematical analysis is a kinetic equation:
Lattice vibrations, responsible for heat transport, are modeled by an interacting gas of phonons whose evolution is described by the Boltzmann Phonon Equation.
Our main result is the rigorous derivation of
an anomalous diffusion equation starting from the {\it linearized} Boltzmann Phonon Equation.
\end{abstract}

\section{Introduction}
The goal of this paper is to investigate some aspects of heat transport in dielectric solid crystals and  to derive a fractional  Fourier-type law for a simple one dimensional model.

At the microscopic level, a solid crystal is composed of atoms that oscillate around their equilibrium positions, which are described by a lattice in $\RR^n$ (which we can take to be $\ZZ^n$ for simplicity).
Denoting by $q_i$ the displacement of the atom $i\in\ZZ^n$ and by $p_i$ its momentum, the dynamic of the crystal is described by a Hamiltonian of the form (we set the mass of the atoms equal to $1$ for simplicity):
$$ 
H(p,q)= \sum_{i\in \ZZ^n}\frac{1}{2} p_i^2 + \sum_{i,j\in \ZZ^n}V(q_{j}-q_i)+ \sum_{i\in \ZZ^n} U(q_i).
$$ 
Because of the coupling (potential $V$) between neighboring atoms, oscillations propagate in the lattice. 
In an electrically insulating crystal, it is these vibrations that are responsible for heat transport (in conductors, other phenomena are at play).
At the macroscopic level, though, heat transport is classically described by Fourier's law, which states that the heat flux $\vec{j}$ behaves as
\begin{equation}\label{eq:Fourier}
\vec{j}=- \kappa\nabla_x T
\end{equation}
where $T$ is the temperature and $\kappa$ is a positive constant that may depend on the temperature itself. 

The relation between these two models, and in particular 
the rigorous derivation of Fourier's law directly from the Hamiltonian dynamic of the atoms (after rescaling of the  space and time variables) is a very challenging and largely open problem (see  \cite{bonetto2000fourier}). 
In this paper, we will consider
an intermediate (mesoscopic) model of kinetic type between the microscopic  and macroscopic levels:
the Boltzmann Phonon Equation.
The idea of using kinetic theory in the study of heat conduction in such crystals 
was first formulated by 
 Debye \cite{Debye} and then Peierls \cite{peierls1929kinetischen}.
One way to think about this kinetic description of heat propagation is to model the vibrations of the lattice as a gas of interacting phonons whose distribution function solves a Boltzmann type equation.
Our goal in this paper is then to derive an anomalous version of Fourier's law from the appropriate linearized Boltzmann phonon equation  (in the same way that the equations of  fluid mechanic have been derived from the Boltzmann equation for diluted gases).

 
The particular framework we are interested in was made popular by a numerical experiment performed by  Fermi, Pasta and Ulam  in the 1950's at Los Alamos National Laboratories. 
The goal of their experiment was to investigate numerically the dynamic (and relaxation toward equilibrium) of the simplest model for a crystal: a chain (dimension $n=1$) of oscillators coupled to their nearest neighbors by non-linear forces  described by a Hamiltonian of the form
$$ H = \sum_{i\in \ZZ}\left[ \frac{1}{2} p_i^2 + V(q_{i+1}-q_i)\right].$$ 
When $V$ is purely harmonic, the system has quasi-periodic solutions and does not relax to an equilibrium  \cite{berman2005fermi} (in that case, we will see below that in the kinetic description, the phonons behave like non-interacting particles). 
Fermi, Pasta and Ulam  thus considered the next two simplest cases by adding a cubic potential $V(r)=\frac{r^2}{8}+\alpha \frac{r^3}{3}$ (this model is now referred to as the FPU-$\alpha$ chain) or a quartic potential $V(r)=\frac{r^2}{8}+\beta \frac{r^4}{4}$ (the FPU-$\beta$ chain).

These models have been widely studied since that original experiment 
as they provide a simple framework in which to understand heat transport in one-dimensional chains (see for instance \cite{livi2005,delfini2007} for some reviews). 
The key feature of these studies is that anomalous heat diffusion is typically the norm for such chains.



The main result of this paper is thus the  derivation of an anomalous diffusion equation for heat transport in the FPU-$\beta$ chain (we will see later why we do not consider the FPU-$\alpha$ chain)
starting from a linearized Boltzmann type kinetic equation.
The first step is to introduce the Boltzmann phonon equation corresponding to the FPU-$\beta$ chain.
While first developed by Peierls \cite{peierls1929kinetischen}, 
the mathematical derivation of this Boltzmann phonon equation starting from the microscopic (Hamiltonian) model has been done formally by H. Spohn in \cite{spohn2006phonon} using Wigner transform. 
We will thus not focus on this step, though we will spend some time in this paper discussing the results of  \cite{spohn2006phonon} in the next section.
Our focus instead will  be on the rigorous derivation of Fourier's law from the linearized Boltzmann phonon equation.
As expected, we will not recover \eqref{eq:Fourier}, but instead a non-local (fractional) version of Fourier law corresponding to an anomalous diffusion equation (in place of the usual heat equation).

\medskip

Let us now describe our main result more precisely.
As mentioned above, the starting point of  our analysis is
the Boltzmann phonon equation given by:
$$
\pa_t W + \om' (k)\pa_x W = C(W)
$$
where the unknown $W(t,x,k)$ is a function of the time $t\geq 0$, the position $x\in\RR$ and the wave number $k\in\TT:=\RR/\ZZ$.
This function is introduced in \cite{spohn2006phonon}
as the Wigner transform of the displacement field of the atoms, but it can be interpreted as a density distribution function for a gas of interacting phonons (describing the chain vibrations).
The function $\omega(k)$ is the dispersion relation for the lattice and 
the operator $C$ describes the interactions between the phonons.

We will discuss in Sections \ref{sec:introo} and \ref{sec:beta} the particular form of $\omega$ and $C$ corresponding to our microscopic models.
For the FPU-$\beta$ chain, the operator $C$ will be the so-called \textit{four phonon collision operator}, which is an integral operator of Boltzmann type but cubic instead of quadratic (see \eqref{eq:CC4}).


As explained above, our goal is to derive a macroscopic equation for the temperature.
This is, at least in spirit, similar to the derivation of fluid mechanics equations from the Boltzmann equation for diluted gas (see \cite{GolseStR} and references therein).  
We will consider a perturbation of a thermodynamical equilibrium $\overline W(k)=\frac{\overline T}{\omega(k)}$ (note that the temperature is classically defined by the relation $E=k_B T$ where $E=\int_\TT \omega(k) W(k)\, dk $ and $k_B$ denotes Boltzmann's constant - here, we choose temperature units so that $k_B=1$):
$$ W^\eps(t,x,k) = \Wb(k)(1+\eps f^\eps(t,x,k)).$$
The function  $f^\eps$ then solves
$$\pa_t f^\eps+ \om'(k)\pa_x f^\eps = L(f^\eps) + \mathcal O(\eps)$$
where $L$ is the linearized operator
$$ L(f)=\frac{1}{\Wb}DC(\Wb)(\Wb f).$$ 

As usual a macroscopic equation is derived after an appropriate rescaling of the time and space variable. More precisely, we will show (see Theorem \ref{thm:main}) that
the solution of 
\begin{equation}\label{eq:ffd}
\eps^{\frac 8 5}\pa_t f^\eps+ \eps \om'(k)\pa_x f^\eps = L(f^\eps)
\end{equation}
converges to a function $T(t,x)$ (so $W^\eps(t,x,k)$ achieves local thermodynamical equilibrium when $\eps\to 0$) solution of
\begin{equation}\label{eq:TTT} 
\pa_t T + \frac{\kappa}{\Tb^{6/5}} (-\Delta)^{\frac 4 5} T=0.
\end{equation}
This fractional diffusion equation corresponds to the anomalous Fourier law (of order $3/5$)
\begin{equation}\label{eq:fourierfrac} \vec{j} = -\kappa(\Tb) \na (-\Delta)^{-\frac 1 5} T.
\end{equation}


\medskip

Before going any further, a discussion seems in order concerning the power $\frac 45$ of the Laplacian that we obtain in \eqref{eq:TTT} and the scaling of \eqref{eq:fourierfrac}. 
Such a scaling is consistent with the earlier theoretical results
of Pereverzev 
 \cite{Pereverzev} and Lukkarinen-Spohn \cite{lukkarinen2008anomalous}.
 This is of course not surprising since these results also relied on the kinetic approximation of the microscopic model.
The same scaling was also found using a different approach (mode coupling theory) in  \cite{lepri1998}.
However, whether or not it is consistent with the scaling of the original Hamiltonian dynamic is still very much under debate.

Indeed, it is important to stress the fact that the Boltzmann-Phonon equation, which serves as the starting point of our rigorous study, is formally derived from the Hamiltonian model as a weak perturbation limit (small quartic perturbation of the harmonic potential). 
In other words, equation \eqref{eq:TTT} is derived through a two steps limiting process 
which might not preserve the scaling of the original dynamic. 
This fact is actually  suggested by the example of the FPU-$\alpha$ chain, 
for which (as we will see later) the kinetic approximation leads to a free transport equation.
Note also that in \cite{NR}, arguments based on a different approach lead to a different  scaling.

Of course this issue has attracted the interest of many scientists and numerous numerical experiments have been performed  since the original work of Fermi, Pasta and Ulam,  \cite{lepri1997,shimada2000simulational,lepri2003thermal,lepri2005studies,delfini2007,GDL10}. But while all those computations clearly show that anomalous heat diffusion is taking place (as opposed to regular diffusion), it is very challenging to determine the precise scaling law.
A complete review of the discussions on this topic is beyond the scope of this paper, 
and we refer the interested reader to the papers cited above (see in particular \cite{lepri2003universality, delfini2007}). 

There is also an important literature devoted to the mathematical analysis of heat propagation in one-dimensional chains using probabilistic techniques.
In this approach, the Hamiltonian dynamics of the microscopic system is described by an harmonic potential but it is  perturbed by a stochastic noise conserving momentum and energy.
In the weak noise limit, one also recovers fractional diffusion phenomena
 (see \cite{basile2006momentum}, \cite{basile2009thermal}, \cite{Olla}, \cite{basile2010energy} and the review paper \cite{olla2009energy}). 
More recently, a similar approach has been developed without the weak noise assumption to derive
fractional heat equations of order $3/4$  (see  \cite{Jara34}  and  \cite{Benardin34}). The same order is obtained in the results presented in \cite{spohn2014nonlinear}, which are applicable to the case of the FPU-$\beta$ chain that we are considering. The approach in \cite{spohn2014nonlinear} is different from ours since the author uses mode coupling theory and the dynamics studied are given by a mesoscopic equation coming from `linear fluctuating hydrodynamics' which are evolution equations with an added noise term.




\medskip

Going back to the object of the present paper, we note that 
the derivation of fractional diffusion equations from kinetic equations such as \eqref{eq:ffd} is now classical (see \cite{Mouhot11}, \cite{mellet2009fractional}, \cite{ben2011anomalous}, \cite{abdallah10Hilbert}, \cite{hittmeir2014kinetic}, \cite{Olla}). As in previous results (see in particular \cite{ben2011anomalous}), the order of the limiting diffusion process is determined by the degeneracy of the collision frequency of the linearized operator $L$.
Our work is thus greatly indebted to the work of J. Lukkarinen and H. Spohn \cite{lukkarinen2008anomalous} who carefully study the properties of the operator $L$ and show in particular that the collision frequency behaves as $|k|^{5/3}$ when $k\to 0$. 
Using this result (and under the so-called {\it kinetic conjecture}), they then show that  
$$\lim_{t\to \infty} t^{3/5}C(t) = c_0$$
for some constant $c_0>0$, where $C(t)$ is the Green-Kubo correlation function (in a standard diffusive setting $C(t)$ would converge to a constant). As we mentioned before, this order is consistent with ours (since a laplacian of order $4/5$ corresponds to an anomalous Fourier's law of order $3/5$).

\medskip

Comparing the result of the present paper with previous results on fractional diffusive limit for kinetic equations mentioned above, we point out that $L$ is the {\it linearized} Boltzmann Phonon operator which is quite different from the {\it linear} Boltzmann operator considered in most previous works. 
In particular, its
kernel is $2$ dimensional, and its cross section ($K(k,k')$ defined by \eqref{eq:KernelK}) changes sign. 
A similar situation arises in \cite{hittmeir2014kinetic}, where the authors consider a linearized BGK-type operator, which does not preserve the positivity of its solution and does conserve more than one quantity. 
In \cite{hittmeir2014kinetic} the authors show that different conserved quantities require different time rescaling in the anomalous diffusive limit. In particular, only the most `singular' quantity shows anomalous diffusion (the rest have constant dynamics). 
The difficulty in our case will thus be to prove that the most singular conserved quantity goes to zero as $\eps$ goes to zero, so that the less singular one (which corresponds to the temperature $T$ above) exhibits anomalous diffusion.

Note that the fact that the kernel is two dimensional (which will be discussed extensively in the next sections) appears to be a mathematical artifact rather than being related to some physical phenomenon.
It does, however, indicate some weakness in the mixing properties of the collision process (this will be even more obvious for the FPU-$\alpha$ chain, for which the collision operator vanishes altogether).
And while the macroscopic behavior of $f^\eps$ is completely determined by the function $T(t,x)$, the other component of the projection of $f^\eps$ onto the kernel of $L$ will play a role in reducing the value of the diffusion coefficient $\kappa$.

\medskip

We will not attempt here to derive a nonlinear Fourier law by working with the nonlinear operator $C$ (rather than the linearized operator $L$).
Such a derivation is developed 
in \cite{bricmont2008approach} by Bricmont and Kupiainen, but under assumptions that ensure that regular diffusion, rather than anomalous diffusion, takes place (non degeneracy of the collision frequency).

\medskip

This paper is organized as follows:
In Section \ref{sec:introo}, we describe the original problem (chains of coupled harmonic oscillators) and 
its relation to the Boltzmann phonon equation. We then introduce the collision operators $C$ that appears in the context of FPU chains. In that section, we will see in particular that this kinetic description cannot be used to study the FPU-$\alpha$ chain because the collision operator $C$ vanishes in that case.
This section is mostly based on the paper of H. Spohn \cite{spohn2006phonon}.
In Section \ref{sec:beta}, we investigate the properties of the four phonon collision operators, appearing in the context of the FPU-$\beta$ chain as well as its linearization around an equilibrium (this section is largely based on the work of J. Lukkarinen and H. Spohn \cite{lukkarinen2008anomalous}). 
The main result of our paper  is finally stated in Section \ref{sec:main_result} and its proof is divided between   Sections \ref{sec:L} and \ref{sec:main}.

\medskip

\section*{Acknowledgement}
This material is based upon work supported by the Kinetic
Research Network (KI-Net) under the NSF Grant No. RNMS  \#1107444.

\noindent A.M. is partially supported by NSF Grant DMS-1201426. 

\noindent S.M thanks the University of Maryland and CSCAMM (Center for Scientific Computation and Mathematical Modeling) for their hospitality; the Cambridge Philosophical Society and Lucy Cavendish College (University of Cambridge) for their financial support. Thanks to Cl\'ement Mouhot from the University of Cambridge for his help and support. Thanks to Herbert Spohn from Technische Universit\"{a}t M\"{u}nchen for useful discussions. S.M is supported by the UK Engineering and Physical Sciences Research Council (EPSRC) grant EP/H023348/1 for the University of Cambridge Centre for Doctoral Training, the Cambridge Centre for Analysis.

\medskip

\section{Crystal vibrations: A kinetic description}\label{sec:introo}
In this section, we recall the results from the paper of H. Spohn \cite{spohn2006phonon} that are relevant to our present study.
Our goal is to detail the relation between the Boltzmann phonon equation that we are considering in this paper and the microscopic models.
At the microscopic level, we  
consider an infinite lattice $\ZZ^n$ describing the equilibrium positions of the atoms of a crystal (we briefly introduce the model in general dimension, though starting in Section \ref{sec:FPU_framework}, we will focus solely on the one-dimensional case).
The deviation of the atom $i\in \ZZ^n$ from its equilibrium position is denoted by $q_i$, and the conjugate momentum variable  is denoted by $p_i$.
We consider the dynamics associated to the Hamiltonian
$$H(q, p) = \frac{1}{2} \sum_{i \in \mathbb{Z}} p_i^2 + V_h(q) + \sqrt\lambda V(q)$$
where $V_h$ is a harmonic potential (quadratic)  and $\sqrt \lambda V$ is a small anharmonic perturbation (the kinetic equation is obtained in the limit $\lambda\to 0$).
The general form of the harmonic potential is
\begin{equation}\label{eq:Vh0} 
V_h (q)=\frac{1}{2} \sum_{i,j\in\ZZ^n} \bar\alpha(i-j)q_iq_j + \frac{\omega_0^2}{2}\sum_{i\in\ZZ^n} q_i^2,
\end{equation}
while $V$ is typically a cubic or quartic potential of the form
$$ V(q) = \sum_{i\in\ZZ^n} \gamma(q_i) \quad \mbox{ or } \quad V(q)=\sum_{\small{\begin{array}{l} i,j\in\ZZ^n \\ |i-j|=1
\end{array}}
}  \gamma(q_j-q_i) .
$$

In order to understand how energy is being transported by the vibration of the atoms in the lattice, we will replace this very large system of ODE by a kinetic equation (the so-called Boltzmann phonon equation) whose unknown $W(x,k,t)$ will be interpreted as a density distribution function for a gas of interacting phonons.  
The idea of describing the lattice vibrations by interacting phonons, whose evolution would be described by a Boltzmann type equation first appeared in a paper of Peierls \cite{peierls1929kinetischen}.
This derivation was made more rigorous by H. Spohn \cite{spohn2006phonon} using Wigner transforms and asymptotic analysis.

We will not give any details concerning this derivation (we refer the interested reader to the work of   H. Spohn \cite{spohn2006phonon}). 
We just claim that (formally at least)
an appropriately rescaled Wigner transform of the displacement field $q$ converges when $\lambda\to 0$ to a function $W(t,x,k)$ solution of the Boltzmann phonon equation
\begin{equation}\label{eq:0}
 \pa_t W + \na_k \om (k)\cdot \na_x W = C(W).
\end{equation}
The function $W$ depends on the time $t\geq 0$, the position $x\in\RR^n$ and a wave vector $k$ which lies in the Torus $\TT^n = \RR^n/\ZZ^n$.
The function $\om(k)$ is the dispersion relation of the lattice. It is determined by the harmonic part of the potential. 
For general 
potential given by  \eqref{eq:Vh0}, we have:
\begin{equation}\label{eq:omegageneral}
 \omega(k)= (\omega_0^2+\widehat{\bar\alpha}(k))^{1/2}
 \end{equation}
where $\widehat {\bar \alpha}(k)$ is the Fourier transform of $a$, defined by
$$ \widehat{\bar\alpha}(k) = \sum_{j\in\ZZ^n}  e^{-i2\pi k\cdot j} \bar\alpha(j).$$
The operator $C$ in the right hand side of \eqref{eq:0} is an integral collision operator which depends on the anharmonic potential $V(q)$. Of course this operator $C$ is crucial in determining the long time behavior of the solutions of this equation, so we will spend a bit of time discussing its properties in this introduction.

\medskip

Note that while the relation between $W(t,x,k)$ and the microscopic variable $q_i$ and $p_i$ is rather complicated,  the total energy of the system is given by
\begin{align} 
\int_{\RR^n}  \int_{\TT^n} \omega(k) W(t,x,k) \, dk\, dx & = \frac{1}{2}\int |\widehat p(k)|^2 +\omega(k)^2 |\widehat q(k)|^2\, dk \nonumber \\
& =\sum_{i\in \ZZ^n} \frac{1}{2} p_i^2 + V_h(q). \label{eq:energydef}
\end{align}

\subsection{The FPU framework} \label{sec:FPU_framework}
As explained in the introduction, we now focus on the FPU chain model.
For this model, we have  $n=1$ (we denote by $\TT$ the torus $\TT=\RR/\ZZ$) and 
the potential describes only nearest neighbors interactions.
The harmonic potential is thus given by:
$$
V_{h} (q)=   \frac{1}{8} \sum_{i\in \mathbb{Z}} (q_{i+1}- q_{i})^2,
$$
and the anharmonic potential $V$ is either cubic (FPU-$\alpha$ chain) or quartic (FPU-$\beta$ chain):
$$ V(q) = \sum_{i\in\ZZ^n} \gamma(q_{i+1}-q_i), \qquad \gamma(q)=\frac 1 3 q^3 \mbox{ or } \gamma(q)=\frac 1 4  q^4.$$
The corresponding microscopic dynamics is given by
\begin{eqnarray}
\frac{d}{dt} q_i(t) &=& p_i(t) \label{eq:hamiltonian_dynamics}\\
\frac{d}{dt} p_i(t) &=&  \frac{1}{4} q_{i+1}(t)-\frac{1}{2} q_i(t)+ \frac{1}{4} q_{i-1}(t) - \sqrt{\lambda} [\gamma'(q_i-q_{i-1}) - \gamma'(q_{i+1}-q_i)] . \nonumber
\end{eqnarray}

\subsection{The dispersion relation}
When $V_h$ is given by
\begin{equation}\label{eq:harmV}
V_{h} (q)= \frac{1}{2} \omega_0^2 \sum_{i\in\ZZ} q_i^2+  \frac{1}{8} \sum_{i\in \mathbb{Z}} (q_{i+1}- q_{i})^2,
\end{equation}
equation \eqref{eq:omegageneral} gives the following formula for 
the dispersion relation:
$$ \omega(k)^2 
=\omega_0^2+ \frac{1}{2}-\frac{1}{4}\lp e^{i2\pi k}+e^{-i 2\pi k}\rp $$
and so 
\begin{equation}\label{eq:om0} 
\omega (k)= \left( \omega_0^2 +  \frac{1}{2}(1-\cos(2\pi k))\right)^{1/2}, \quad k\in\TT.
\end{equation}
For the FPU model, we have $\omega_0=0$, and so the dispersion relation is given by
$$ \omega(k)=\sqrt{\frac{1}{2}(1-\cos(2\pi k))}=|\sin(\pi k) |.$$

\subsection{The interaction operator $C$}
The operator $C$ in the right hand side of \eqref{eq:0} is determined by 
the non-harmonic perturbation of the potential $V$.

\paragraph{Cubic potentials: Three phonons operator}
When the anharmonic potential is cubic, that is
\begin{equation}\label{eq:Vcub1} 
V=\frac 1 3 \sum_{i\in\ZZ} q_i^3,
\end{equation}
or
\begin{equation}\label{eq:Vcub2}
 V=\frac 1 3 \sum_{i\in\ZZ} (q_{i+1}-q_i)^3
\end{equation}
(the latter one corresponds to the FPU-$\alpha$ chain), 
 the collision operator is given by
 \begin{align} 
 C(W) = & \, 4\pi \int\int F(k,k_1,k_2)^2
 \nonumber \\
 & \times \Big[
 2\delta(k+k_1-k_2) \delta( \omega+\omega_1-\omega_2) (W_1W_2+WW_2-WW_1) \nonumber \\
 &\; +  \delta(k-k_1-k_2) \delta( \omega-\omega_1-\omega_2) (W_1W_2-WW_1-WW_2) \Big]
  dk_1 dk_2\label{eq:C3}
 \end{align}
where we used the notation $\omega_i=\omega(k_i)$ and $W_i=W(k_i)$.

The formula for the collision rate $F(k,k_1,k_2)$ can be found in \cite{spohn2006phonon}.
In particular, when $V$ is given by \eqref{eq:Vcub1} (on-site potential) then 
$$ 
F(k,k_1,k_2)^2 =  (8\omega \omega_1\omega_2)^{-1}
$$
When $V$ is the nearest neighbor interaction potential \eqref{eq:Vcub2}
and  $\omega_0=0$ (that is for the FPU-$\alpha$ chain), the collision rate becomes
$$ F(k,k_1,k_2)^2 =  (8\omega \omega_1\omega_2)^{-1}|  [\exp(i2\pi k)-1] [\exp(i2\pi k_1)-1] [\exp(i2\pi k_2)-1]|^2 .$$
 Using the fact that 
 $$ |\exp(i2\pi k)-1|^2 = 4\sin^2(\pi k),$$
 we find
 $$ F(k,k_1,k_2)^2=8\frac{\sin^2(\pi k)\sin^2(\pi k_1)\sin^2(\pi k_2)}{\omega \omega_1\omega_2}=8\omega\omega_1\omega_2.
 $$

Going back to \eqref{eq:C3}, we note that
the first term can be interpreted as describing a wave vector $k$ merging with a wave vector $k_1$ and leading to a new wave vector $k_2$ ($k+k_1\to k_2$), while the second term describes the splitting of wave vector $k$ into $k_1$ and $k_2$ ($k\to k_1+k_2$). See Figure \ref{fig:1}.
 \begin{figure}
\begin{center}
\scalebox{1}{ \pdfimage{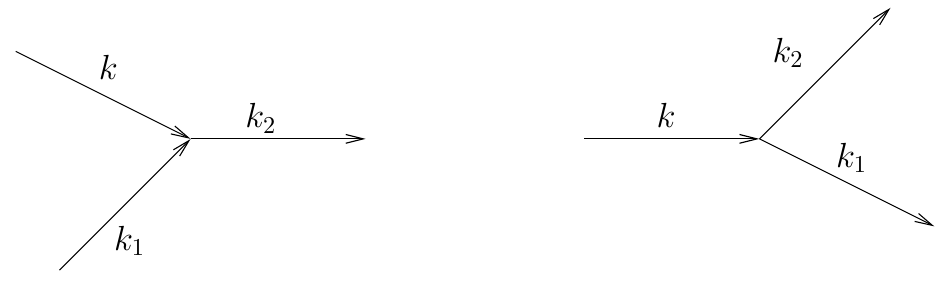}}
\end{center}
\caption{Three phonons interactions}\label{fig:1}
\end{figure} 
These interactions conserve the energy ($\omega+\omega_1=\omega_2$), but the momentum is conserved only modulo integers: the $\delta$-function in the first term yields the constraint $k+k_1=k_2 +n$, $n\in\ZZ$, $k,k_1,k_2\in\TT$ (one talks of normal process when $n=0$, and {\it umklapp} process when $n\neq 0$).

This quadratic operator is reminiscent of the Boltzmann operator for the theory of dilute gas. There is however an essential difference: The kinetic energy $\frac 1 2 v^2$ is replaced here by the dispersion relation $\omega(k)$. 
In order to further study this integral operator, it is thus essential to characterize the set of $(k,k_1,k_2)$ such that the $\delta$-functions are not zero, that is:
$$\left\{
\begin{array}{l}
k+k_1=k_2 \\
\omega(k)+\om(k_1)=\om(k_2)
\end{array}
\right.
$$
or
\begin{equation}\label{eq:3phonon}
\omega(k)+\omega(k_1)=\omega(k+k_1), \qquad (k,k_1)\in \TT,
\end{equation}
This is much more delicate than for the usual Bolzmann operator
and for general dispersion relation $\omega$, it is not obvious that \eqref{eq:3phonon} has any solutions.

In our framework, that is when $\omega$ is given by \eqref{eq:om0}  (nearest neighbor harmonic coupling) we actually can prove (\cite{spohn2006phonon}) that
$$ \omega(k)+ \omega(k_1) - \omega(k+k_1) \geq \frac{\omega_0}{2} $$
so \eqref{eq:3phonon} has no solutions when $\omega_0>0$.
When $\omega_0=0$, we can check easily that the equation
$$  \omega(k)+ \omega(k_1) = \omega(k+k_1)$$
only has the trivial solution $k_1=0$ and that the equation
$$  \omega(k)= \omega(k_1) + \omega(k-k_1)$$
only has the trivial solutions $k_1=0$ and $k_1=k$. Using the explicit formula for $C$, we can then check that this lead to $C(W)=0$ for all functions $W$.
We thus have:
\begin{theorem}
When $\omega$ is given by \eqref{eq:om0} with $\omega_0\geq 0$, then the three phonon collision operator \eqref{eq:C3} satisfies $C(W)=0$ for all $W$.
\end{theorem} 

\medskip

In particular, this implies that for the FPU-$\alpha$ chain, the collision operator vanishes, and  the corresponding Boltzmann phonon equation reduces to pure transport.
This suggests  poor relaxation to equilibrium for the microscopic model, and it means that this kinetic approach is of no use in studying the long time behavior of the hamiltonian system.
This is of course the reason why we focus in this paper on the FPU-$\beta$ chain.

\medskip

\begin{remark} \label{rmk:1}
As noted in \cite{spohn2006phonon},  equation \eqref{eq:3phonon} might have non trivial solutions for other dispersion relations (for instance 
$\omega(k)=\omega_0+2(1-\cos(2\pi k))$), so this three phonon operator is of interest in other frameworks (different harmonic potential $V_h$).
\end{remark}

\medskip
\medskip

 \paragraph{Quartic potentials: Four phonons operator.}
We now consider the quartic  potential
given by
\begin{equation}\label{eq:Vquar1} 
V(q)=\frac{1}{4}\sum_{i\in \ZZ} q_i^4
\end{equation}
or
\begin{equation}\label{eq:Vquar2} 
V=\frac 1 4 \sum_{i\in\ZZ} (q_{i+1}-q_i)^4.
\end{equation}

The corresponding collision operator then reads
 \begin{align} 
 C(W) = &\,12\pi   \sum_{\sigma_1,\sigma_2,\sigma_3=\pm1} \int\int\int 
 F(k,k_1,k_2,k_3) ^2
  \nonumber \\
 & \times \delta(k+\sigma_1k_1+\sigma_2k_2+\sigma_3 k_3) \delta( \omega+\sigma_1\omega_1+\sigma_2\omega_2+\sigma_3 \omega_3)\nonumber \\
 & \times(W_1W_2W_3+W(\sigma_1W_2W_3+W_1\sigma_2W_3+W_1W_2\sigma_3)) \, dk_1\, dk_2\, dk_3\label{eq:C4}
 \end{align}
 with
 $$ F(k,k_1,k_2,k_3)^2 = (16 \omega \omega_1\omega_2\omega_3)^{-1}$$
 for on-site potential \eqref{eq:Vquar1}
 and 
 \begin{equation}\label{eq:colfreq}
F(k,k_1,k_2,k_3)^2 = \prod_{i=0}^3 \frac{2\sin^2(\pi k_i)}{\om(k_i)}=16\omega\omega_1\omega_2\omega_3.
\end{equation}
for nearest neighbor coupling \eqref{eq:Vquar2}.

The term proportional to $W$ is the loss term, while the gain term is $W_1W_2W_3$ (which is always positive).
Again, we can interpret  the different terms as pair collisions or merging/splitting of phonons (see Figure \ref{fig:2}). 
 \begin{figure}
\begin{center}
\scalebox{1}{ \pdfimage{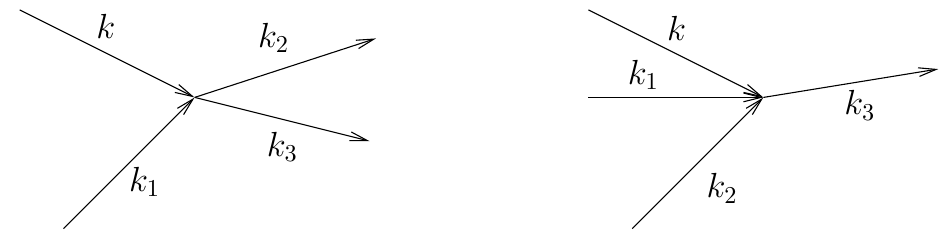}}
\end{center}
\caption{Four phonons interactions}\label{fig:2}
\end{figure} 
In order to understand the collision rule, we note that for pair collisions $(k,k_1)\to(k_2,k_3)$ (which correspond to the terms such that $\sum_{j=1}^3 \sigma_j =-1$ in the integral), we need to solve
\begin{equation}\label{cons1}
\omega(k)+\omega(k_1) = \omega(k_2) + \omega(k+k_1-k_2)
\end{equation}
while for three phonons mergers (or splitting) $(k,k_1,k_2)\to k_3$ we have
\begin{equation}\label{cons2} 
\omega(k) + \omega(k_1) +\omega(k_2) =\omega(k+k_1+k_2).
\end{equation}

In general, it is not possible to solve these equations explicitly, and it is not obvious that 
either of these equations should be satisfied on a set of positive measure.
In fact, when $\omega$ is given by \eqref{eq:om0} (nearest neighbor  couplings), it can be shown (see \cite{spohn2008}) that
\eqref{cons2} has no solution when $\omega_0>0$, 
and only the trivial solutions (in which two of $k$, $k_1$ and $k_2$ vanish) when $\omega_0=0$.
In this latter case, a careful accounting of all the terms in \eqref{eq:C4} corresponding to three phonons split or merger  ($\sigma_1\sigma_2\sigma_3=-1$) show that these terms cancel out. 
As a consequence,  the only interactions that are contributing to the operator $C$ are pair collisions, which, in particular, preserve the total number of phonons.
This preservation of the number of phonons, reminiscent of the preservation of the number of particles in gas dynamics, does not follow here from a fundamental physical principle, but is instead a mathematical artifact
(note that this property holds also for the nonlinear wave equation for which $\omega(k)=|k|$, $k\in \RR^3$).

As a consequence, the operator $C$ can be  rewritten as
\begin{align}
C(W)& = 36\pi\!\int\!\int\!\int \!F(k,k_1,k_2,k_3)^2
\delta(k+k_1-k_2-k_3) 
\delta(\om+\om_1-\om_2-\om_3) \nonumber \\
& [W_1W_2W_3+WW_2W_3-WW_1W_3-WW_1W_2] \, dk_1\, dk_2\, dk_3 .\label{eq:CC4}
\end{align}

When $\omega$ is given by \eqref{eq:om0}, we will see later on that \eqref{cons1} has non trivial solutions on a set of full measure, that is
$$ \int_\TT \int_\TT \delta(\omega(k)+\omega(k_1) - \omega(k_2) - \omega(k+k_1-k_2) )\, dk_1\, dk_2 \neq 0.$$
In particular this operator $C$ is non trivial.

\section{FPU-$\beta$ chain: The four phonon collision operator}\label{sec:beta}
In this section, we briefly summarize the properties of the four phonon collision operator \eqref{eq:CC4} which arises in the modeling of the FPU-$\beta$ chain.

\subsection{Conserved quantities}
All the collision operators $C$ mentioned above conserve the energy. This can be expressed by the following condition:
$$ \int_\TT \omega(k) C(W)(k)\, dk = 0 $$
for all functions $W$.

The four phonon collision operator \eqref{eq:CC4}, corresponding to the quartic potential, also satisfies
$$ \int_\TT C(W)(k)\, dk = 0 $$
which can be interpreted as the conservation of the number of phonons $\int_\TT W\, dk$. 
However, this quantity has no microscopic equivalent, and does not correspond to any physical principle. Rather it is a consequence of the symmetry of the operator, 
which follows from the fact that $3$ phonon merger cannot take place (\eqref{cons2} has no solutions).
In particular, this equality does not hold for the three phonon operator. 
\medskip

Note that the first moment $k$ is preserved in the wave kinetic equation case (where $k\in\RR^n$). However, this conservation is broken here by umklapp processes.

\subsection{Entropy}
The Boltzmann phonon operators satisfy an entropy inequality, similar to Boltzmann H-Theorem in gas dynamic.
In particular, for the four phonon operator we can rewrite \eqref{eq:CC4} as follows:
\begin{eqnarray}
C(W)& = & 36\pi\int\int\int
F(k)^2
\delta(k+k_1-k_2-k_3) 
\delta(\om+\om_1-\om_2-\om_3) \nonumber \\
&&WW_1W_2W_3 [W^{-1}+W_1^{-1}-W_2^{-1}-W_3^{-1}] dk_1\, dk_2\, dk_3  \nonumber
\end{eqnarray}
and  we then see that (assuming all integrals are well defined):
\begin{eqnarray} \label{eq:Hthm}
&& \int_{\TT^1} W^{-1}(k)C(W)(k)\, dk 
\\
&& \quad \qquad =  9\pi\int \int\int\int
F(k)^2
\delta(k+k_1-k_2-k_3) 
\delta(\om+\om_1-\om_2-\om_3) \nonumber \\
 &&\quad \quad\qquad \qquad \cdot WW_1W_2W_3 [W^{-1}+W_1^{-1}-W_2^{-1}-W_3^{-1}] ^2 dk_1\, dk_2\, dk_3\, dk \nonumber\\
 && \quad \qquad \geq 0 .\nonumber
\end{eqnarray}
This inequality implies in particular that the total entropy defined as
$$H(W) = \int_\mathbb{R}\int_\TT \log(W)\, dkdx$$
increases over time.

\subsection{Stationary solutions}
It is easy to check that the distributions 
$$ W_b (k) = \frac{1}{b \omega(k)}$$
for any $ b > 0$
satisfy $C(W_b)=0$ for all the operators $C$ considered above. This fact is in accordance with equilibrium statistical mechanics (see \cite{spohn2006phonon}).
It is  more delicate to check that these are the only solutions. 
In fact it is not always true. 
\medskip

For the  four phonon collision operator  \eqref{eq:CC4}, 
we can check that 
\begin{equation}\label{eq:Wab}
W_{a,b} (k) = \frac{1}{a+b\omega(k)}
\end{equation}
is an equilibrium for all $a$, $b>0$.

Conversely, the entropy inequality  \eqref{eq:Hthm} implies that if $C(W)=0$ then $\psi(k)=W(k)^{-1}$ is a collision invariant, that is  
$$\psi(k)+\psi(k_1)=\psi(k_2)+\psi(k_3)$$
for all $k,k_1,k_2,k_3$ such that
$$  k+k_1=k_2+k_3,\quad\mbox{ and } \quad \omega(k)+\omega(k_1) = \omega(k_2)+\omega(k_3).$$

An obvious candidate is $\psi(k)=a+b \omega(k)$. 
Under general conditions on $\omega$, Spohn  proved that these are indeed the only collision invariants in dimension $N\geq 2$ \cite{spohn2006collisional}. 
The same result is proved by Lukkarinen and Spohn \cite{lukkarinen2008anomalous} in our framework (dimension $1$).

As a conclusion, \eqref{eq:Wab} are the only solutions of $C(W)=0$ for the four phonon collision operator  \eqref{eq:CC4}.
Note that the fact that we can take $a\neq 0$  is a consequence of the conservation of the number of phonons for the four phonon collision operator (which, as explained previously, follows  from the fact that equation \eqref{cons2} describing merging and splitting of phonons has no solutions).

\subsection{The linearized operator}
As mentioned in the introduction, we will be interested in the behavior of the solutions of the Boltzmann phonon equation in the neighborhood of a thermodynamical equilibrium.
Given $\Wb(k)=\frac{\Tb}{\om(k)}$, we thus introduce the linearized operator
$$ L(f) = \frac{1}{\Wb} DC(\Wb)(\Wb f)$$
(where $DC$ denotes the derivative of the operator $C$).

By differentiating the equation $C(W_{a,b})=0$ with respect to $a$ and $b$, we get:
$$ L(1)=0 \quad \mbox{ and } L(\om^{-1})=0,$$
which suggests (as will be proved later) that the kernel of $L$ is two dimensional and spanned by $1$ and $\om^{-1}$. 
In our framework, the later mode, $\om^{-1}$ is singular (not integrable) for $k=0$.
Because of natural a priori bounds on the solutions of the Boltzmann Phonon   equation, it will be easy to see that this mode is not present in the macroscopic limit. It will however play an important role in the derivation of a macroscopic model. Note that it comes from the derivation with respect to the spurious coefficient $a$ (remember that we call it spurious because it comes from the conservation of the total number of phonons at the level of the kinetic equation;  we know that actually this conservation does not take place at atomic level and it is just a consequence of losing information in the kinetic limit).

Similarly, differentiating the conservation equations
$$\int \omega C(\Wb+t \Wb f)\, dk = 0 \mbox{ and } \int   C(\Wb+t \Wb f)\, dk = 0$$
with respect to $t$, we deduce that
$$  \int L(f) \, dk = 0 ,  \mbox{ and } \int \omega^{-1} L(f)\, dk = 0.$$

The properties of $L$ will be further investigated in Section \ref{sec:L}. For now, we just state the following proposition without proof, since it is all we need to formally derive a macroscopic equation.

\begin{proposition}\label{prop:formalL}
The operator $L: L^2(\TT^1,V(k)\,dk)\tto L^2(\TT^1,V(k)^{-1}\,dk)$ (where $V$ is defined by \eqref{eq:coefV}) is a bounded self-adjoint operator which satisfies
\item[(i)] $\ker (L) = \mbox{Span}\,\{1,\omega(k)^{-1}\}$
\item[(ii)] $R(L) = \{ h\in L^2(\TT^1,V(k)^{-1}\,dk ) \, ;\, \int_\TT h(k) \, dk = \int_\TT \omega^{-1}(k) h(k)\, dk =0\,\}$
\end{proposition}

We end this section by deriving the explicit formula for the operator $L$:
A direct computation gives (when $W(k)=\frac{\Tb}{\om(k)}$):
\begin{align*}
& DC(W)(W f)  \\
& = 36\pi\int\int\int F(k,k_1,k_2,k_3)^2
\delta(k+k_1-k_2-k_3) 
\delta(\om+\om_1-\om_2-\om_3) \nonumber \\
& \qquad \times W W_1W_2W_3\Big [ 
f_3W_3^{-1}+ f_2W_2^{-1} -f_1W_1^{-1} - fW^{-1}\Big ] \, dk_1\, dk_2\, dk_3 \\
& = 36\pi \Tb^3 \int\int\int \frac{F(k,k_1,k_2,k_3)^2}{\omega\omega_1\omega_2\omega_3}
\delta(k+k_1-k_2-k_3) 
\delta(\om+\om_1-\om_2-\om_3) \nonumber \\
&\qquad\times \Big [ 
 \omega_3 f_3+ \omega_2f_2 -\omega_1 f_1 - \omega f\Big ] \, dk_1\, dk_2\, dk_3 
\end{align*}
Using  \eqref{eq:colfreq}, we see that
$$
\frac{F(k,k_1,k_2,k_3)^2}{\omega\omega_1\omega_2\omega_3} = 16$$
and we deduce:
\begin{align} 
L(f)& = 576 \pi \Tb^2 
 \omega \int\int\int \delta(k+k_1-k_2-k_3) 
\delta(\om+\om_1-\om_2-\om_3) \nonumber \\
&\quad\qquad  \qquad\qquad \qquad \times \Big [  \omega_3 f_3+ \omega_2f_2 -\omega_1 f_1 - \omega f\Big ] \, dk_1\, dk_2\, dk_3 . \label{eq:L000}
\end{align}

\subsection{Formal asymptotic limit}
We now have all the ingredient to perform the usual asymptotic analysis and attempt to derive (formally) a diffusion equation from the Boltzmann phonon equation (we will see however that it fails in our framework). 
The starting point is the following rescaled equation in the FPU-$\beta$ chain framework detailed above:
\begin{equation}\label{eq:kin0}
\eps^2 \pa_t W +\eps  \om'(k)\pa_x W = C(W),
\end{equation}
where $C$ is the four phonon collision operator \eqref{eq:CC4} with collision frequency given by
\eqref{eq:colfreq},
and 
we consider a solution which is a perturbation of a thermodynamical equilibrium:
$$ W^\eps(t,x,k)= \Wb(k) (1+\eps f^\eps(t,x,k))
$$
where $\Wb=\frac{\Tb}{ \omega(k)}$ for some constant $\Tb>0$.

We introduce the operators
$$  Q(f,f)=\frac{1}{\Wb} D^2C(\Wb)(\Wb f,\Wb f),$$
and 
$$ R(f,f,f)  =\frac{1}{\Wb} D^3C(\Wb)(\Wb f,\Wb f,\Wb f)$$
so that (we recall that $C$ is a cubic operator):
$$ \frac{1}{\Wb} C(W^\eps) = \eps L( f) + \eps^2\frac{1}{2}Q (f, f) + \eps^3 \frac{1}{6} R(f,f,f)$$
where $L$ is given by \eqref{eq:L000}.

The function $f^\eps$ solves
\begin{equation}\label{eq:linear}
\eps^2 \pa_t f^\eps + \eps \omega'(k) \pa_x f^\eps =   L(f^\eps) + \eps\frac 1 2 Q(f^\eps,f^\eps) + \eps^2\frac 1 6  R(f^\eps,f^\eps,f^\eps).
\end{equation}

Taking the limit $\eps\to0$ in (\ref{eq:linear}), we get 
$$ L(f^0)=0$$
and so Proposition \ref{prop:formalL} $(i)$ implies
$$f^0(t,x,k) = T(t,x) + S(t,x)\omega(k)^{-1}.$$
Since equation \eqref{eq:linear} preserves the $L^1$ norm, it is natural to assume that $f^0(t,x,k)\in L^1(\RR\times \TT)$. We note however  that  $\om(k)\sim |k|$ as $|k|\to 0$, and so we must have 
$$ S(t,x)=0.$$

Next, integrating (\ref{eq:linear}) with respect to $k$ yields
$$ 
\pa_t T^\eps +  \pa_x J^\eps= 0 
$$
with
$$ T^\eps= \la f^\eps\ra, \quad J^\eps(t,x)=\frac{1}{\eps}\la \om' f^\eps\ra
$$
where we use the notation $\la \cdot \ra = \int_\TT \cdot \, dk$.

We now need to compute $J = \lim_{\eps\to 0 } J^\eps$.
Recalling that $L$ is a self adjoint operator, we write
$$ 
 \eps^{-1}  \la \om' f^\eps\ra = \eps^{-1}\la L^{-1}(\om') L(f^\eps)\ra 
$$
and using (\ref{eq:linear}), we replace  $L(f^\eps)$ in the right hand side:
$$
 \eps^{-1}  \la \om' f^\eps\ra = \la L^{-1}(\om') \om' \pa_x f^\eps \ra  -  \frac{1}{2}\la L^{-1}(\om' ) Q(f^\eps,f^\eps) \ra + \mathcal O(\eps).
$$
Formally, using the fact that $\lim_{\eps\to0} f^\eps(t,x,k) = T(t,x)$ (since $S=0$),
we thus get
$$ \lim_{\eps\to 0 } \eps^{-1}  \la \om' f^\eps\ra =  \la L^{-1}(\om') \om'\ra\pa_x T  -  \frac{1}{2}\la L^{-1}(\om' ) Q(T,T) \ra. $$
Finally, a direct computation gives
\begin{align*} 
Q(f,f)& = 576 \pi \Tb^2 
 \omega \int\int\int \delta(k+k_1-k_2-k_3) 
\delta(\om+\om_1-\om_2-\om_3) \nonumber \\
&\quad\qquad   \Big [  
2(\om-\om_3)[f_1f_2-ff_3]
+(\om+\om_1)[f_2f_3-ff_1]
\Big ] \, dk_1\, dk_2\, dk_3, 
\end{align*}
and it is readily seen that $ Q(T,T)=0$.
We thus get the following relation
$$ J =  \la L^{-1}(\om') \om'\ra\pa_x T $$
which is Fourier's law with 
diffusion coefficient
$$ \kappa=- \la L^{-1}(\om') \om'\ra > 0.$$

We conclude this section with the following remarks:
\begin{enumerate}
\item  The non linear term $ Q(T,T)=0$ does not contribute to the limiting equation. In the next  section, we will drop this term altogether.
\item The fact that $S=0$ will need to be addressed very carefully in the rigorous proof. In particular, we will see that while we do indeed have $f^0=T$, the term $S$ plays a significant role in the rigorous derivation of the diffusion equation (see next section). 
\item Perhaps the most important remark is that one need to check that $\kappa$ is well defined.
In fact, it can be proved that the integrand in the definition of the diffusion coefficient behaves like $|k|^{-5/3}$ for small $k$. It follows that
 $$ \kappa=+\infty$$
so the limit presented above does not give any equation for the evolution of $T$.
Such a phenomenon is not uncommon, and based on previous work (see \cite{Mouhot11}), we expect that by taking a different time scale in \eqref{eq:linear} we can derive an anomalous diffusion equation for the evolution of the temperature $T$. This is of course the goal of this paper as explained in the next section.
\end{enumerate}

\section{Main result}
\label{sec:main_result}

In view of the formal asymptotic limit detailed in the previous section, we now consider the following linear equation:
\begin{equation}\label{eq:linear2}
\eps^\alpha \pa_t f^\eps + \eps \omega'(k) \pa_x f^\eps = \Tb^2 L(f^\eps), \qquad x\in\RR, \; k\in \TT
\end{equation}
where 
$$\omega(k)=|\sin(\pi k)|$$ 
and
$L$ is defined by
\begin{align} 
L(f)& =  
 \omega \int\int\int \delta(k+k_1-k_2-k_3) 
\delta(\om+\om_1-\om_2-\om_3) \nonumber \\
&\quad\qquad   \Big [  \omega_3 f_3+ \omega_2f_2 -\omega_1 f_1 - \omega f\Big ] \, dk_1\, dk_2\, dk_3 . \label{eq:L}
\end{align} 
Note also that we have made $L$ independent of the equilibrium temperature $\Tb$ and 
set all other constant in $L$ equal to $1$ for the sake of clarity.

The existence of a solution to this equation is fairly classical. We recall it for the sake of completeness in Proposition \ref{prop:Cauchyproblem}.

Our main result is then the following: 
\begin{theorem}[Fractional diffusion limit for the linearised equation]\label{thm:main}
Take $\alpha=\frac 8 5$ and  let $f^\eps$ be a solution of equation \eqref{eq:linear2} with initial data $f_0\in L^2 (\RR\times\TT)$. Then
$$f^\eps(t,x,k) \rightharpoonup T(t,x)  \qquad L^\infty((0,\infty); L^2(\RR\times \TT))\mbox{-weak}\,*$$
where $T$ solves the fractional diffusion equation
\begin{equation}\label{eq:difffT}
\partial_t T +\frac{\kappa}{\Tb^{6/5}}(-\Delta_x)^{4/5}T=0 \qquad \mbox{ in } (0,\infty)\times\RR
\end{equation}
with initial condition
\begin{equation}\label{eq:initT}
T(0,x) =T_0(x):=  \int^1_0  f_0(x,k) \,  dk.
\end{equation}
The diffusion coefficient $\kappa\in(0,\infty)$ is given by 
$$\kappa = \kappa_1-\frac{\kappa_2^2}{\kappa_3}\in(0,\infty)$$
where $\kappa_1, \kappa_2, \kappa_3$ are defined in Proposition \ref{prop:aeps1}.
\end{theorem}

First, we note that it is enough to consider the case 
$$\Tb=1$$
since we can recover the general case by a simple rescaling $t\mapsto \Tb^2 t$, $x\mapsto \Tb^2 x$.

The main difficulty here, compared with previous work devoted to fractional diffusion limit of kinetic equations, is the fact that the kernel of $L$ is spanned by $1$ and $\omega(k)^{-1}$. This last mode should not appear in the limit since it is not square integrable, but it will nevertheless play an important role.

In fact, we will prove that $f^\eps$ can be expanded as follows:
$$ f^\eps (t,x,k) = T^\eps(t,x) + \eps^{\frac{3}{5}} S^\eps(t,x) \omega(k)^{-1}  + \eps^{\frac 4 5 } h^\eps(t,x,k)$$
where $T^\eps$ is bounded in $L^\infty((0,\infty);L^2(\RR))$, $h^\eps$ is bounded in $L^2_V(\TT\times\RR)$ and  $S^\eps$  converges in some weak sense to a non trivial function.
More precisely we will prove in Section \ref{sec:Seps}:
\begin{proposition}\label{prop:convS}
The function $S^\eps(t,x)$ converges in distribution sense to
$$  S(t,x)  =- \frac{\kappa_2}{\kappa_3} (-\Delta)^{3/10} T(t,x).$$
\end{proposition}

In particular, as mentioned above, this means that the mode $\omega(k)^{-1}$ vanishes in the limit and the macroscopic behavior of the phonon distribution is completely described by $T=\lim_{\eps\to 0}T^\eps$.
However, projecting equation \eqref{eq:linear2} onto the constant mode of the kernel of $L$, we will find the following equation of the evolution of $T$:
\begin{equation}\label{eq:diffT}
\pa_t T + \kappa_1 (-\Delta)^{4/5} T + \kappa_2 (-\Delta)^{1/2}S =0.
\end{equation}
We see that $S=\lim_{\eps \rightarrow 0} S^\eps$ plays a role in the evolution of $T$.
To understand this, we note (anticipating a bit on the result of the next section) that  the reason we are observing anomalous diffusion phenomena here (as opposed to standard diffusion as described in the previous section), is the fact that phonons with frequency $k$ close to zero encounter very few collisions (degenerate collision frequency).
And the term $ \eps^{\frac{3}{5}} S^\eps(t,x) \omega(k)^{-1}$, while small, is heavily concentrated around $k=0$ (non integrable singularity at $k=0$).
The competition between the smallness and the singularity gives rise to a term of order $1$ in the equation.

In order to describe the evolution of $T$, we now need to obtain an equation for $S$.
By projecting equation \eqref{eq:linear2} onto the $\omega(k)^{-1}$ mode of the kernel of $L$, we will 
prove that:
\begin{equation}\label{eq:diffS} 
 \kappa_2 (-\Delta)^{1/2} T + \kappa_3 (-\Delta)^{1/5}S = 0 .
 \end{equation}
We note that there is no $\pa_t S$ in \eqref{eq:diffS} (unlike the corresponding equation for $T$). The reason is that due to the singularity of $\omega(k)^{-1}$ for $k=0$, the quantity $S$ diffuses faster than $T$ (so we would have to take a smaller $\alpha$ in \eqref{eq:linear2} in order to observe the diffusion of $S$).  At our time scale (given by $\alpha  = \frac 8 5$), $S$ has thus already reached equilibrium, and can be expressed (in view of \eqref{eq:diffS}) as
$$  S  =- \frac{\kappa_2}{\kappa_3} (-\Delta)^{3/10} T .$$
Inserting this expression into \eqref{eq:diffT}, we find
$$ \pa_t T +  \kappa (-\Delta)^{4/5} T  =0$$
where
$\kappa= \kappa_1-\frac{\kappa_2^2}{\kappa_3}$. 
Of course, we will show that $\kappa>0$ (once the explicit expressions for the $\kappa_i$ are given, it will be a very simple consequence of Cauchy-Schwarz inequality - see Lemma \ref{lem:kappapos}). It is interesting to note that the effect of the mode $\omega^{-1}$ on the macroscopic equation is to reduce the diffusion coefficient (and thus to slow down the diffusion).
This can be understood by noting that the fact that the kernel of $L$ does not contain only the natural constant mode, is due to the lack of merging $k+k_1+k_2 \to k_3$ and splitting $k\to k_1+k_2+k_3$ interactions for phonons in the non linear collision operator $C$ (fewer interactions $\Rightarrow$ slower relaxation).

\section{Properties of the operator $L$}\label{sec:L}
The asymptotic behavior of the solution of \eqref{eq:linear2} depends very strongly on the properties of the operator $L$.
This operator is studied in great detail in  \cite{lukkarinen2008anomalous}, and we will recall their main results in this section.

The operator $L$ can be written as
$$ L(f)=\int K(k,k') f(k')\, dk' - V(k)f(k)$$
where
\begin{align}
K(k,k') &  = \om(k)\om(k') \int_\TT  \; 2 \, \delta(\om(k)+\om(k_1)-\om(k')-\om(k+k_1-k'))\nonumber \\
& \qquad\qquad \quad \qquad -\delta(\om(k)+\om(k')-\om(k_1)-\om(k+k'-k_1))\, dk_1\label{eq:KernelK}
\end{align}
and 
\begin{equation}\label{eq:coefV}
V(k)=\om(k)^2 \int_{\TT^2 } \delta(\om(k)+\om(k_1)-\om(k')-\om(k+k_1-k'))\, dk_1 dk'.
\end{equation}
The fact that $\int_{\TT} L(f)\,dk = 0 $ for all $f$ implies 
$$V(k)=\int_\TT K(k',k)\, dk'$$
(this equality can be checked also from the formula for $K$ and $V$, but it is much easier this way)
and
a short computation shows that 
$$K(k,k') = K(k',k).$$
In particular, $L$ is a self adjoint operator in $L^2(\TT)$ and $-L$ is positive since we have
\begin{eqnarray} 
- \int_{\TT} L(f)f\, dk & = & 
\frac1 4  \int\int\int\int
\delta(k+k_1-k_2-k_3) 
\delta(\om+\om_1-\om_2-\om_3) \nonumber \\
&& [\om_3 f_3 +\om_2 f_2 -\om_1 f_1 -\om f]^2 \,dk\,  dk_1\, dk_2\, dk_3 \label{eq:Lp}\\
&\geq& 0\nonumber 
\end{eqnarray}
for all $f$.
One of our goals will be to improve this inequality and show that $L$ has a spectral gap property in the appropriate functional spaces.
For that, we will need to show that the integral operator
\begin{equation}\label{eq:KV}
K(f)=\int K(k,k') f(k')\, dk'
\end{equation}
is a compact operator (in an appropriate functional spaces)

The first step, in view of  \eqref{eq:KernelK} is to study the solution set of the equation of conservation of energy:
\begin{equation}\label{eq:consenergy1}
\om(k)+\om(k_1)=\om(k')+\om(k+k_1-k').
\end{equation}
We recall the following result:
\begin{proposition}[\cite{lukkarinen2008anomalous}]
The equation \eqref{eq:consenergy1}
has the trivial solutions $ k'=k$ and $ k'=k_1$, and the (non trivial) solution
$$ k_1 = h(k,k')$$
where 
$$h(k,k')=\frac{k'-k}{2}+2\arcsin\left(\tan\frac{|k'-k|}{4}\cos\frac{k+k'}{4}\right)$$
(and there are no other solutions of  \eqref{eq:consenergy1}).
\end{proposition}

With this proposition in hand, one can now compute the kernel $K(k,k')$ and the multiplicative function $V(k)$.  
We recall here the main result of \cite{lukkarinen2008anomalous}. The first one 
states that the function $V(k)$ is degenerate for $k\to 0$ (note that $W$ in \cite{lukkarinen2008anomalous}  corresponds to our $V$):
\begin{proposition}[{\cite[Lemma~4.1]{lukkarinen2008anomalous}}]
\label{prop:degeneracy_multiplicative_operator}
The function $V:\RR\to\RR_+$ is symmetric ($V(1-k)=V(k)$), continuous and satisfies
\begin{equation}\label{eq:V} 
c_1|\sin(\pi k)|^{5/3} \leq V(k)\leq c_2|\sin(\pi k)|^{5/3}
\end{equation}
for all $k\in\RR$, for some $c_1, c_2>0$. Moreover,
$$\lim_{k\rightarrow 0} \lp \left| \sin\pi k \right|^{-5/3} V(k) \rp = v_0 >0.$$
\end{proposition}

Because of the degeneracy of $V$ for $k=0$,  we do not expect the operator $L$ to have a spectral gap in $L^2$.
We thus introduce the operator  
$$L_0(f): = V^{-1/2}L(V^{-1/2}f)$$
We note that this operator has the form
$$L_0(f)=K_0(f)-f$$
with
$$K_0(f)=V^{-1/2}K (V^{-1/2}f).$$

To prove that $L_0$ has good properties in $L^2(\TT)$, we need to study the operator $K_0$.
Again, it is proved in  \cite{lukkarinen2008anomalous} that $K_0:L^2(\TT^1)\to L^2(\TT^1)$ is a 
compact, self-adjoint operator, which implies that $K:L^2(\TT^1,V\, dk)\to L^2(\TT^1,V^{-1}dk)$  is a compact, self-adjoint operator.

To be more precise, in \cite{lukkarinen2008anomalous}, the kernel $K$ is first written as
$$K(k,k')=2\omega(k) K_2(k,k')\omega(k')-\omega(k) K_1(k,k')\omega(k')$$
where 
\begin{equation} \label{eq:k1k2}
K_1(k,k'):=4 \frac{\mathbb{1}\lp F_-(k,k')>0\rp}{\sqrt{F_-(k,k')}}\quad \mbox{ and} \quad K_2(k,k'):= \frac{2}{\sqrt{F_+(k,k')}}
\end{equation}
for $k,k' \in [0,1]$ and
$$F_{\mbox{\textpm}}(k,k') = \lp \cos( \pi k) +\cos(\pi k' )\rp^2 \mbox{\textpm}4\sin(\pi k) \sin( \pi k').$$
and the main result of \cite{lukkarinen2008anomalous} is the following:
\begin{proposition}[{\cite[Propositions 4.3 and 4.4.]{lukkarinen2008anomalous}}]
\label{prop:properties_integral_kernels} 
Let $\psi:[0,1] \rightarrow \mathbb{R}$ be given, and assume that there are $C,p>0$ such that
$$|\psi(k)|\leq C \lp \sin \pi k \rp^p$$
for all $k\in[0,1]$. Then the kernels
$$  \psi(k)^* K_2(k,k') \psi(k') \qquad \mbox{ and } \quad \psi(k)^* K_1(k,k') \psi(k')$$
define  compact, self-adjoint integral operators in $L^2(\TT)$.
\end{proposition}

We immediately conclude:
\begin{corollary}
The  kernel
\be \label{eq:K0}
K_0(k,k') = V^{-1/2}(k) \omega(k) \lp 2 K_2(k, k') -K_1(k,k')\rp \omega(k') V^{-1/2}(k')
\ee
defines a compact self-adjoint operator in $L^2((0,1))$.
As a consequence, the kernel 
$$ K(k,k') = V^{1/2}(k) K_0(k,k') V^{1/2}(k')$$
defines a compact self-adjoint operator  from $L^2(\TT^1,V(k)\,dk)$ onto  $L^2(\TT^1,V(k)^{-1}\,dk)$.
In particular,
\begin{equation}\label{eq:bdK1}
\int_\TT |K(f)(k)|^2 V(k)^{-1}\, dk \leq C \int_\TT |f(k)|^2  V(k)\, dk .
\end{equation}
\end{corollary}
\begin{proof}
Indeed, by Proposition \ref{prop:degeneracy_multiplicative_operator} we have that
$$
V^{-1/2}(k) \omega(k)  \leq  c_2 \lp \sin \pi k\rp^{1/6} 
$$
and 
the claim follows from Proposition \ref{prop:properties_integral_kernels}.
\end{proof}

Furthermore, we note that we have not used the full potential of Proposition \eqref{prop:properties_integral_kernels}. 
We can thus improve \eqref{eq:bdK1} as follows: 
\begin{corollary}
The  kernel
$$
\widetilde{K_0}(k,k'):= \lp\sin( \pi k) \rp^{-1/6 +\eta} K_0(x,k') \lp \sin (\pi k') \rp^{-1/6+\eta}\quad \eta>0
$$
defines a compact self-adjoint operator in $L^2((0,1))$.
In particular, for all $\eta>0$, there exists $C(\eta)$ such that
\begin{equation}\label{eq:bdK}
\int_\TT |K(f)(k)|^2 (\sin(\pi k))^{-\frac 1 3+\eta} V(k)^{-1}\, dk \leq C \int_\TT |f(k)|^2 (\sin(\pi k))^{\frac 1 3-\eta} V(k)\, dk 
\end{equation}
\end{corollary}
\begin{proof}
Using Proposition \ref{prop:degeneracy_multiplicative_operator} we have that
\beqar
V^{-1/2}(k) \omega(k) \lp \sin \pi k \rp^{-1/6+\eta} &\leq&  c_2 \lp \sin \pi k\rp^{1/6} \lp \sin \pi k \rp^{-1/6 +\eta}\\
&=& c_2 \lp \sin \pi k \rp^\eta
\eeqar
the claim follows from Proposition \ref{prop:properties_integral_kernels}.
\end{proof}

\medskip

We have thus showed that $L:L^2(\TT^1,V(k)\,dk)\tto L^2(\TT^1,V(k)^{-1}\,dk)$ was a bounded operator.
Next, we characterize the kernel of $L$:
First, we note that given $f\in L^2(\TT^1,V(k)\,dk)$, 
inequality \eqref{eq:Lp} implies that if $L(f)=0$ then
\begin{align*}
& \int\!\! \int\!\! \int \!\! \int
\delta(k+k_1-k_2-k_3) 
\delta(\om+\om_1-\om_2-\om_3)  \\
& \qquad\qquad\qquad\qquad
\times [\om_3 f_3 +\om_2 f_2 -\om_1 f_1 -\om f]^2 \,dk\,  dk_1\, dk_2\, dk_3 =0.
\end{align*}
So $f$ must satisfy
$$ \om(k) f(k) +  \om(k_1) f(k_1)  =  \om(k_2) f(k_2)  + \om(k+k_1-k_2) f(k+k_1-k_2) $$
whenever
$$ \om(k)  +  \om(k_1)  =  \om(k_2)   + \om(k+k_1-k_2). $$
We also say that $\omega(k)f(k)$ must be a collision invariant.
Such invariants have been characterized in \cite{lukkarinen2008anomalous}:
\begin{theorem}[\cite{lukkarinen2008anomalous}]\label{thm:LS}
A function $\psi\in L^1(\TT)$ is a collisional invariant if and only if there exists $c_1$ and $c_2$ such that
$$ \psi(k)=c_1+c_2\omega(k).$$
\end{theorem}
As a consequence, we deduce:
\begin{corollary}
The kernel of $L$ is the two dimensional subspace of $L^2(\TT^1,V(k)\,dk)$ spanned by the functions $  1$ and $\omega(k)^{-1}$
 (note that both of those functions belongs to $L^2(\TT^1,V(k)\,dk)$ thanks to (\ref{eq:V}))
\end{corollary}

\medskip

Finally, the compactness of $K$ and inequality \eqref{eq:Lp} implies
\begin{lemma}\label{lem:coer}
There exists $c_0>0$ such that
$$ - \int_{\TT^1} L(f)f\, dk \geq c_0\int V(k) |f-  \Pi( f)|^2\, dk$$
for all $f\in L^2(\TT^1,V(k)\,dk)$, where $\Pi(f)$ denotes the orthogonal projection of $f$ onto $\ker(L)$.
\end{lemma}

\medskip

To summarize, we have thus showed:
\begin{proposition}\label{prop:L}
The operator $L: L^2(\TT^1,V(k)\,dk)\tto L^2(\TT^1,V(k)^{-1}\,dk)$ is bounded
and satisfies:
\begin{enumerate}
\item The kernel of $L$ has dimension $2$ and is spanned by $1$ and $\frac{1}{\omega(k)}$.
\item For all $f\in L^2(\TT^1,V(k)\,dk)$, we have
\begin{equation}\label{eq:Lcons} \int_{\TT^1} L(f)\, dk=0 \quad \mbox{ and }\quad \int_{\TT^1} \frac{1}{\om(k)} L(f)\, dk =0.
\end{equation}
\item There exists $c_0>0$ such that
$$ - \int_{\TT^1} L(f)f\, dk \geq c_0\int V(k) |f-  \Pi( f)|^2\, dk$$
for all $f\in L^2(\TT^1,V(k)\,dk)$, where $\Pi(f)$ denotes the orthogonal projection of $f$ onto $\ker(L)$.
\end{enumerate}
\end{proposition}

Note that the projection of $f$ onto $\ker(L)$ can be written as
$$\Pi(f)=T + S \left[\la V\ra \om(k)^{-1}- \la V\om^{-1}\ra \right] $$
with
$$ T =  \frac{1}{\la  V\ra } \int V(k) f(k) \, dk\mbox{ and } S = \frac{1}{m_0} \int \left[\la V\ra  \frac{V(k)}{\omega(k)} - \la V\om^{-1}\ra  V(k)\right] f(k)\, dk $$
where $m_0=\la V\ra^2 \la V \om^{-2}\ra -\la V\om^{-1}\ra^2 \la V\ra$ is a normalization constant. 
The operator  $\Pi$ is a continuous operator in $L^2(V(k)\, dk)$.

\medskip
We finish this section commenting on the existence of solutions for the equation for the sake of completeness:
\begin{proposition}[Cauchy Problem] \label{prop:Cauchyproblem}
There exists a unique solution in $L^\infty((0,\infty); L^2(\RR\times \TT))$ for equation \eqref{eq:linear2} with initial data $f_0\in L^2(\RR\times\TT)$.
\end{proposition}
\begin{proof}
A traditional method for solving the  Cauchy problem for this type of equations uses an iterative scheme based on the mild formulation: 
$$f(t,x,k) = f_0 (x-\omega'(k)t,k) + \int^t_0 Lf(x-(t-s) \omega'(k), s) ds$$
together with the estimate
$$\|L(f)\|_{L^2(\RR\times\TT)} \leq C\|f\|_{L^2(\RR \times \TT)}.$$
This last estimate is consequence of \eqref{eq:bdK1} and the boundedness of the function $V$. 
We refer to \cite{allaire2013transport} and \cite{mouhotlecturenotes} for further details on this method. 
\end{proof}

\section{Proof of Theorem \ref{thm:main}} \label{sec:main}

\subsection{A priori estimates}
As a first step in the proof of Theorem \ref{thm:main}, we establish some a priori estimates.
The coercivity property of $L$ (Lemma \ref{lem:coer}) gives the following proposition:
\begin{proposition}\label{prop:apriori}
Assume that $f_0\in L^2(\RR\times\TT)$.
Then, the function $f^\eps(t,x,k)$, solution of \eqref{eq:linear2} satisfies
\begin{equation}\label{eq:L2f} 
||f^\eps(t) || _{L^2(\RR\times\TT)}\leq ||f_0||_{L^2(\RR\times\TT)}\qquad \mbox{ for all } t\geq 0.
\end{equation}
Furthermore, $f^\eps$  can be expanded as follows:
\begin{equation}\label{eq:feps}
 f^\eps = \Pi (f^\eps) + \eps^{4/5} h^\eps,
 \end{equation}
where 
\begin{equation}\label{eq:h} 
\| h^\eps\| _{L^2_V((0,\infty)\times \RR\times\TT)} \leq C ||f_0||_{L^2(\RR\times\TT)}
\end{equation}
and  $\Pi (f^\eps )$ is the projection of $f^\eps$ onto $\ker(L)$, given by
$$ \Pi (f^\eps)(t,x,k)  = \tilde T^\eps (t,x)+\tilde S^\eps(t,x)\left[\la V\ra \om(k)^{-1}- \la V\om^{-1}\ra \right] $$
with
\begin{eqnarray} \nonumber
\tilde T^\eps(t,x) &=&  \frac{1}{\la  V\ra } \int V(k) f^\eps (t,x,k) \, dk\, ,\\
\tilde S^\eps(t,x) &=& \frac{1}{m_0} \int \left[\la V\ra  \frac{V(k)}{\omega(k)} - \la V\om^{-1}\ra  V(k)\right] f^\eps(t,x,k)\, dk  \label{eq:definitionStilde}
\end{eqnarray}
where $\tilde T^\eps, \tilde S^\eps$ are bounded in $L^\infty((0,\infty); L^2(\RR))$.
\end{proposition}
\begin{proof}
Multiplying \eqref{eq:linear2} by $f^\eps$ and integrating with respect to $x$ and $k$, we get 
$$
\frac{1}{2}\frac{d}{dt}\|  f^\eps(t)\| ^2_{L^2(\RR\times\TT^1)} - \frac{1}{\eps^\alpha} \int_{\RR} \int_{\TT^1} L(f^\eps)f^\eps\, dk\, dx =0.
$$
Integrating with respect to $t$ and using Lemma \ref{lem:coer}, we deduce
$$
\frac{1}{2}\|  f^\eps(t)\| ^2_{L^2(\RR\times\TT^1)} +\frac{c_0}{\eps^\alpha} \int_0^t\int_{\RR} \int_{\TT^1} V(k) |f^\eps- \Pi( f^\eps)|^2\, dk\, dx\, ds \leq \frac{1}{2}\|  f^\eps_0||^2_{L^2(\RR\times\TT^1)}.
$$
which implies the proposition. The fact that $\tilde T^\eps, \tilde S^\eps \in L^\infty((0,\infty); L^2(\RR))$ is a direct consequence of this estimate and Cauchy-Schwartz.
\end{proof}

Because the singular terms in  $\Pi (f^\eps)$ (those involving $\omega(k)^{-1}$) play a particular role in the sequel, we will prefer to write $\Pi(f^\eps)$ as follows:
$$ \Pi f^\eps =  T^\eps + \frac{\la V\ra }{\omega} \tilde S^\eps(x,t)  $$
with
$$ T^\eps(t,x)=  \tilde T^\eps (t,x)- \tilde S^\eps(t,x) \la V\om^{-1}\ra 
$$
Finally, we set
\begin{equation}\label{eq:Sepsdef}
S^\eps (t,x)= \eps^{-3/5} \la V\ra \tilde S^\eps(t,x),
\end{equation}
leading to the following expansion of $f^\eps$:
\begin{equation} \label{eq:fexpansion}
 f^\eps (t,x,k) = T^\eps(t,x) + \eps^{\frac{3}{5}} S^\eps(t,x) \omega(k)^{-1}  + \eps^{\frac 4 5 } h^\eps(t,x,k).
 \end{equation}
Note that while $T^\eps$ and $h^\eps$ are clearly bounded (in appropriate functional spaces) in view of Proposition \ref{prop:apriori}, the scaling of $S^\eps$ may seem arbitrary at this point. 
However, we will see later on that $S^\eps$ defined as in \eqref{eq:Sepsdef} indeed converges to a non trivial function (in some weak sense).

\subsection{Laplace Fourier Transform}
As in \cite{Mouhot11}, the main tool in deriving the macroscopic equation for $T$ is the use of the Laplace-Fourier transform.
More precisely, we define
$$ \widehat {f^\eps} (p,\xi,k) = \int_{\RR}\int_0^\infty e^{-pt}e^{-i\xi  x} f^\eps(t,x,k)\, dt\, dx.$$
We also denote by $\widehat f_0(\xi,k)$ the Fourier transform of $f_0(x,k)$.

\begin{remark}\label{rem:laplacenorm}
We recall that the Fourier transform preserves the $L^2(\RR)$ norm (Parseval's theorem).
It is also easy to see that the Laplace transform of an $L^1$ function is in $L^\infty$. However our functions are not $L^1$ with respect to $t$. Instead, we will make use of the simple fact that for a given function $g(t)$, its Laplace transform $\widehat g(p)$ satisfies
\begin{equation}\label{eq:laplaceL^2}
 | \widehat g(  p) | \leq  \frac{1}{p}\| g\|_{L^\infty (0,\infty)}\quad  \mbox{ and }\quad 
 | \widehat g(  p) | \leq \frac{1}{\sqrt{2p}}\| g\|_{L^2(0,\infty)}\end{equation}
for all $p> 0$.
\end{remark}

Taking the Laplace Fourier transform of Equation \eqref{eq:linear2} (with $\overline T=1$), we obtain:
$$
\eps^\alpha p \hfe - \eps^\alpha \widehat{f_0} + i\eps \omega'(k) \xi \hfe = K(\hfe)-V\hfe
$$
which easily yields
\begin{equation}\label{eq:ll} 
\hfe(p,\xi,k) = \frac{ \eps^\alpha }{\eps^\alpha p+V(k)+ i\eps \omega'(k) \xi } \widehat{f_0} + \frac{ 1 }{\eps^\alpha p+V(k)+ i\eps \omega'(k) \xi } K(\hfe).
\end{equation}
We recall that $L(f)=K(f)-Vf$ with $K(f)=\int K(k,k') f(k')\, dk'$. 
The fact that $\int L(f)\, dk=0$ and $\int\frac{1}{\omega(k)}L(f)\, dk=0$ for all $f$ implies
$$ V(k) = \int K(k',k) dk', \qquad \frac{V(k)}{\omega(k)}= \int K(k',k)\frac{1}{\omega(k')}\, dk'$$

Multiplying \eqref{eq:ll} by $K(k',k)$ and integrating with respect to $k$ and $k'$, we get
\begin{align*} 
\int_\TT K(\hfe)(k')dk' &  = \int_{\TT}\int_{\TT} \frac{ \eps^\alpha K(k',k) }{\eps^\alpha p+V(k)+ i\eps \omega'(k) \xi } \widehat{f_0}(\xi,k) \, dk  \, dk'\\
& \quad + \int_\TT \int_{\TT} \frac{ K(k',k) }{\eps^\alpha p+V(k)+ i\eps \omega'(k) \xi } K(\hfe)(k)\, dk\,dk' \\
&  = \int_{\TT} \frac{ \eps^\alpha V(k) }{\eps^\alpha p+V(k)+ i\eps \omega'(k) \xi } \widehat{f_0}(\xi,k) \, dk  \\
&\quad  + \int_{\TT} \frac{ V(k) }{\eps^\alpha p+V(k)+ i\eps \omega'(k) \xi } K(\hfe)(k)\, dk.
\end{align*}
We deduce
\begin{align}
0 & = \int_{\TT} \frac{ V(k) }{\eps^\alpha p+V(k)+ i\eps \omega'(k) \xi } \widehat{f_0}(\xi,k) \, dk  \nonumber \\
&  \quad +   \eps^{-\alpha} \int_{\TT} \left( \frac{ V(k) }{\eps^\alpha p+V(k)+ i\eps \omega'(k) \xi } -1\right) K(\hfe)(k)\, dk.\label{eq:symb1}
\end{align}

Similarly, multiplying \eqref{eq:ll} by $K(k',k)\frac{\eps^{\frac{3}{5}}}{\omega(k')}$, and
 we get:
\begin{align}
0 & = \eps^{\frac{3}{5}} \int_{\TT} \frac{   V(k) }{\eps^\alpha p+V(k)+ i\eps \omega'(k) \xi } \frac{\widehat{f_0}(\xi,k)}{\omega(k)} \, dk  \nonumber \\
&  \quad +   \eps^{-\alpha} \eps^{\frac{3}{5}}  \int_{\TT} \left( \frac{ V(k) }{\eps^\alpha p+V(k)+ i\eps \omega'(k) \xi } -1\right)  \frac{K(\hfe)(k)}{\omega(k)}  \, dk.\label{eq:symb2}
\end{align}

Next, we write
$$ K(\hfe) = K(\Pi(\hfe)) + K(\hfe-\Pi(\hfe)) = V \Pi(\hfe) + K(\hfe-\Pi(\hfe))$$
where we rewrite
$$  \Pi(\hfe) = \widehat T^\eps + \eps^{3/5} \frac{1}{\omega(k)} \widehat S^\eps.$$
We can thus rewrite \eqref{eq:symb1} as follows:
\begin{equation} \label{eq:asymptotic1}
\mathcal F_1^\eps(\widehat f^0)+ a_1^\eps(p,\xi) \widehat T^\eps(p,\xi) + a_2^\eps(p,\xi) \widehat S^\eps(p,\xi) + R_1^\eps(p,\xi) = 0
\end{equation}
and \eqref{eq:symb2} as follows: 
\begin{equation}\label{eq:asymptotic2}
\mathcal F_2^\eps(\widehat f^0)+ a_2^\eps(p,\xi) \widehat T^\eps(p,\xi) + a_3^\eps(p,\xi) \widehat S^\eps(p,\xi) + R_2^\eps(p,\xi) = 0,
\end{equation}
where for $\alpha=8/5$, we have:
\begin{align*}
\mathcal F_1^\eps(\widehat f^0) & = \int_{\TT} \frac{ V(k) }{\eps^{\frac{8}{5}} p+V(k)+ i\eps \omega'(k) \xi } \widehat{f_0}(\xi,k) \, dk  \\
\mathcal F_2^\eps(\widehat f^0) &  =  \eps^{\frac{3}{5}} \int_{\TT} \frac{   V(k) }{\eps^{\frac{8}{5}} p+V(k)+ i\eps \omega'(k) \xi } \frac{\widehat{f_0}(\xi,k)}{\omega(k)} \, dk,
\end{align*}
\begin{align*}  
a_1^\eps(p,\xi)& := \eps^{-\frac{8}{5}} \int_{\TT} \left( \frac{ V(k) }{\eps^{\frac{8}{5}} p+V(k)+ i\eps \omega'(k) \xi } -1\right)V(k)\, dk\\
a_2^\eps(p,\xi)& := \eps^{-\frac{8}{5}} \int_{\TT} \left( \frac{ V(k) }{\eps^{\frac{8}{5}} p+V(k)+ i\eps \omega'(k) \xi } -1\right)\frac{\eps^{\frac{3}{5}}V(k)}{\omega(k)}\, dk\\
& =\eps^{-1} \int_{\TT} \left( \frac{ V(k) }{\eps^{\frac{8}{5}}p+V(k)+ i\eps \omega'(k) \xi } -1\right)\frac{V(k)}{\omega(k)} \, dk\\
a_3^\eps(p,\xi)& := \eps^{-1} \int_{\TT} \left( \frac{ V(k) }{\eps^{\frac{8}{5}} p+V(k)+ i\eps \omega'(k) \xi } -1\right)\frac{\eps^{\frac{3}{5}}V(k)}{\omega(k)^2}\, dk
\end{align*}
and 
\begin{align*}
R_1^\eps(\xi,p) &:=  \eps^{-\frac{8}{5}} \int_{\TT} \left( \frac{ V(k) }{\eps^{\frac{8}{5}} p+V(k)+ i\eps \omega'(k) \xi } -1\right) K(\hfe-\Pi(\hfe))(k)\, dk\\
R_2^\eps(\xi,p) &:= \eps^{-1} \int_{\TT} \left( \frac{ V(k) }{\eps^{\frac{8}{5}} p+V(k)+ i\eps \omega'(k) \xi } -1\right) \frac{1}{\omega(k)}K(\hfe-\Pi(\hfe))(k)\, dk
\end{align*}

In order to prove the main theorem, we now need to pass to the limit in \eqref{eq:asymptotic1} and \eqref{eq:asymptotic2}.
The following three propositions, which are proved in the next section, give  the necessary results for that.

First, we have the following limits for the terms involving the  initial data:
\begin{proposition}\label{prop:initial}
The following limits hold for all $P\geq 0$:
\begin{align*}
 \mathcal F_1^\eps(\widehat f^0)(\xi,p) & \tto \int_\TT \widehat f^0(\xi,k)\, dk = \widehat T_0(\xi) &  \mbox{ in } L^2( (0,P)\times\RR) \\
 \mathcal F_2^\eps(\widehat f^0) (\xi,p) & \tto 0   & \mbox{ in } L^1( (0,P)\times\RR)
 \end{align*}
when $\eps\to0$.
\end{proposition}

Next, we pass to the limit in the symbol $a_i^\eps(p,\xi)$:
\begin{proposition}\label{prop:aeps1}
The following limits hold pointwise $(p,\xi)\in (0,\infty)\times\RR$ and strongly in $L^p_{loc}((0,\infty)\times\RR)$ for all $p\in(1,\infty)$:
\begin{align} 
a_1^\eps(p,\xi)&  \tto 
-p-\kappa_1 |\xi|^{\frac{8}{5}} \quad \mbox{ with } \quad \kappa_1 =  \frac{6}{5} \left( \frac{\pi}{v_0}\right)^{3/5}    \int_{0}^{\infty}    \frac{ z^{3/5}}{ z^2+1 }  \, dz\\
a_2^\eps(p,\xi)&  \tto - \kappa_2 |\xi| \quad \qquad \mbox{
with }\kappa_2=\frac{6}{5} \int_0^\infty \frac{1}{z^2+1}\, dz \\
a_3^\eps(p,\xi)&  \tto -\kappa_3 |\xi|^{\frac{2}{5}} \quad \qquad  \mbox{ with } 
\kappa_3=\frac{6}{5}\left( \frac{v_0}{\pi}\right)^{3/ 5}\int_0^\infty \frac{z^{-3/ 5}}{z^2+1}\, dz
\end{align}
Furthermore, $a^\eps_1, a^\eps_2, a^\eps_3 \in L^\infty_{loc}((0,\infty)\times \RR)$ uniformly with respect to $\eps$.
\end{proposition}
Finally, we need to show that the remainder terms, involving $f^\eps-\Pi(f^\eps)$, go to zero: 
\begin{proposition}\label{prop:remainder}
For all $0<a<P$ and $K>0$, we have 
$$R^\eps_i\rightarrow 0 \quad \mbox{in } L^2((a,P)\times (-K,K))$$
as $\eps \rightarrow 0$ for $i=1,2$. 
 \end{proposition}

\begin{proof}[Proof of Theorem \ref{thm:main}]
We are now ready to prove Theorem \ref{thm:main}.
First, using Proposition \ref{prop:apriori}, we see that up to a subsequence, $T^\eps(t,x)$ converges weakly to $T(t,x)$ in $L^2((0,\tau)\times\RR)$ for all $\tau$ (the uniqueness of the limit will give the convergence of the whole sequence).

Next, for a given test function $\vphi(p,\xi)$ in $\mathcal D((0,\infty)\times\RR)$, we then have
\begin{equation}\label{eq:L2weak} 
\int_0^\infty\int_\RR \widehat T^\eps(p,\xi) \vphi(p,\xi) \, d\xi\, dp =  \int_0^\infty\int_\RR T^\eps(t,x) \widehat  \vphi(t,x) \, dx\, dt
\end{equation}
where $\widehat  \vphi \in L^2((0,\infty)\times\RR)$. This last fact is the classical  Parseval inequality for the Fourier transform, while  for the Laplace transform,  it follows from Minkowski's integral inequality:
\begin{align*}
\left( \int_0^\infty \left( \int_0^\infty  e^{-pt} \vphi(p)\, dp\right)^2 \, dt \right)^{1/2}
&  \leq 
\int_0^\infty \left( \int_0^\infty e^{-2pt}\, dt \right)^{1/2}\vphi(p)\, dp\\
& \leq \int_0^\infty\frac{1}{\sqrt{2p}} \vphi(p)\, dp <\infty.
\end{align*}
 Thus $\widehat T^\eps$ converges to $\widehat T$ in $\mathcal D'((0,\infty)\times\RR)$. Since $\widehat T^\eps$ is also bounded in $L^2_{loc}((0,\infty)\times\RR)$ (using \eqref{eq:laplaceL^2}), 
we deduce that (up to another subsequence) it converges weakly  in $L^2_{loc}((0,\infty)\times\RR)$ to $\widehat T$.


\medskip

In order to derive the equation satisfied by $\widehat T$,  
we need to pass to the limit in \eqref{eq:asymptotic1} and \eqref{eq:asymptotic2}. However, we do not know that $S^\eps$ (defined  in \eqref{eq:Sepsdef}) is bounded in some functional space.
So we multiply equation \eqref{eq:asymptotic1} by $a^\eps_3$ and \eqref{eq:asymptotic2} by $a^\eps_2$ and consider their difference, in order to get rid of the terms in $\widehat S^\eps$:
\begin{eqnarray*}
0&=&a_3^\eps(p,\xi)\mathcal F_1^\eps(\widehat f^0)+ \left( a_3^\eps(p,\xi) a_1^\eps(p,\xi) - \left(a_2^\eps(p,\xi)\right)^2 \right)\widehat T^\eps(p,\xi)  \\
&&+\, a_3^\eps(p,\xi) R_1^\eps(p,\xi) -a_2^\eps(p,\xi)\mathcal F_2^\eps(\widehat f^0) -a_2^\eps(p,\xi) R_2^\eps(p,\xi). \nonumber
\end{eqnarray*} 
Using Proposition \ref{prop:aeps1}, Proposition \ref{prop:remainder} and Proposition \ref{prop:initial}, we can now pass to the limit in this equation in $\mathcal D'((0,\infty)\times\RR)$ and deduce:
$$-\kappa_3|\xi|^{2/5}\widehat T_0+\left(-\kappa_3 |\xi|^{2/5}(-p-\kappa_1|\xi|^{8/5}) -\kappa_2^2 |\xi|^2 \right) \widehat T=0 \quad \mbox{ in } \mathcal D'((0,\infty)\times\RR).$$
Furthermore, factorizing $-\kappa_3 |\xi|^{2/5}$ in this last equation we get
$$-\kappa_3|\xi|^{2/5}\left(\widehat T_0-p\widehat T-(\kappa_1-\frac{\kappa_2^2}{\kappa_3})|\xi|^{8/5} \widehat T\right)=0 \quad \mbox{ in } \mathcal D'((0,\infty)\times\RR).$$
This implies that the function
\begin{equation}\label{eq:finallimit}
g(p,\xi):=\widehat T_0-p\widehat T-\left(\kappa_1-\frac{\kappa_2^2}{\kappa_3} \right) |\xi|^{8/5}\widehat T,
\end{equation}
which belongs to $L^2_{loc}((0,\infty)\times\RR)$, satisfies
$$g(p,\xi)=0 \mbox{ a.e. in } (0,\infty)\times \RR$$
which gives \eqref{eq:difffT}-\eqref{eq:initT}.

\medskip

To complete the proof of Theorem \ref{thm:main},
it remains to show that $f^\eps$ converges to $T(t,x)$ (weakly in $L^\infty((0,\infty),L^2(\RR\times\TT))$).
Since  $f^\eps$ is bounded in $L^\infty(0,\infty;L^2(\RR\times\TT)$, and in view of the expansion \eqref{eq:fexpansion}, it is enough to show that $\eps^{3/5} S^\eps$ converges to zero in some weak sense.

This follows from Proposition \ref{prop:convS}, the proof of which  uses equation \eqref{eq:asymptotic2} and some bounds from below on $a_3^\eps(p,\xi)$ and  will be detailed  in Section~\ref{sec:Seps}.
\end{proof}

 We end this section by proving that the diffusion coefficient $\kappa$ is indeed positive:
\begin{lemma}\label{lem:kappapos}
The coefficients $\kappa_1$, $\kappa_2$ and $\kappa_3$ are such that
$$ \kappa_1 - \frac{\kappa_2^2}{\kappa_3} >0.$$
\end{lemma}
\begin{proof}
Indeed, this is equivalent to 
$$ \kappa_2^2 < \kappa_1 \kappa_3$$
and 
using the explicit formula for  $\kappa_1$, $\kappa_2$ and $\kappa_3$, we see that this is equivalent to
$$ \left(\int_0^\infty \frac{1}{1+z^2}\, dz \right)^2 < \int_0^\infty \frac{z^{3/5}}{1+z^2}\, dz\, \int_0^\infty \frac{z^{-3/5}}{1+z^2}\, dz
$$
which is an immediate consequence of H\"older inequality. 
\end{proof}

\subsection{Proofs of the asymptotic results}
We recall here that $\TT$ denotes the torus $\RR/\ZZ$ and that $\omega(k)=|\sin(\pi k)|$.
Since the dispersion relation $\omega$ is degenerate at $k=0 \pm n$, it will be easier in the computation below to work with $k$ in the symmetric interval $(-\frac 1 2,\frac1 2 )$ (when working with the interval $(0,1)$, we have to deal with both endpoints $0$ and $1$).
Note that the function $\omega$ is even in that interval and that
$$ \omega'(k) = \sgn(k) \pi \cos (\pi k).$$
Finally, Proposition \ref{prop:degeneracy_multiplicative_operator} implies:
\begin{proposition}\label{eq:W}
The function $k\mapsto V(k)$ is even and non-negative on the interval $(-\frac 1 2,\frac1 2)$.
Furthermore the function  $W(k) := V(k)|k|^{-5/3}$ for $k\in(-\frac 1 2,\frac1 2)$
satisfies
$$ \lim_{k\to 0 } W(k) = w_0 := v_0 \pi^{5/3}$$
and 
$$ C_0^{-1}\leq W(k)\leq C_0$$
for some $C_0>0$.
\end{proposition}

\begin{proof}[Proof of Proposition \ref{prop:initial}]
The first part of the proposition follows immediately from Lebesgue dominated convergence theorem, since
$$ \left| \frac{ V(k) }{\eps^{\frac 8 5} p+V(k)+ i\eps \omega'(k) \xi } \right|  = \frac{ V(k) }{\sqrt{ \left(\eps^{\frac 8 5} p+V(k)\right)^2+ (\eps \omega'(k) \xi)^2} } \leq 1$$
and
$$\frac{ V(k) }{\eps^{\frac 8 5} p+V(k)+ i\eps \omega'(k) \xi }  \tto 1 \quad\mbox{ as } \eps\to0.$$

For the second part, we note that 
$$ \left| \frac{ V(k) }{\eps^{\frac 8 5} p+V(k)+ i\eps \omega'(k) \xi } \right|  \leq \frac{ V(k) }{\eps^{\frac 8 5} p+V(k)} $$
and so 
\begin{align} \nonumber
|\mathcal F_2^\eps(\widehat f^0)|(\xi,p)  &  \leq   \eps^{\frac{3}{5}} \int_{\TT}  \frac{ V(k) }{\eps^{\frac 8 5} p+V(k)} \frac{\widehat{f_0}(\xi,k)}{\omega(k)} \, dk\\ \nonumber
&  \leq  C  \eps^{\frac{3}{5}} ||\widehat{f_0}(\xi,\cdot)||_{L^\infty(\TT)}  \int_{0}^{1/2}  \frac{ |k|^{2/3} }{\eps^{\frac 8 5} p+|k|^{5/3}}  \, dk\\
& \leq C \eps^{\frac{3}{5}} ||\widehat{f_0}(\xi,\cdot)||_{L^\infty(\TT)}  (1+|\ln(\eps^{\frac 8 5} p)|) \label{eq:estimateF2}
\end{align}
hence the result, since this last inequality implies (integrating with respect to $\xi$ and $p$)
$$|| \mathcal F_2^\eps(\widehat f^0) ||_{L^1((0,P)\times\RR)} \leq  C \eps^{\frac{3}{5}} ||f_0||_{L^\infty(\RR\times \TT)} P  (1+ |\ln(\eps^{\frac 8 5} P )|).
$$
\end{proof}

\begin{proof}[Proof of Proposition \ref{prop:aeps1}]
First, we write
\begin{align}
 1-\frac{ V(k) }{\eps^{\frac{8}{5}} p+V(k)+ i\eps \omega'(k) \xi }  = &  \frac{ \eps^{\frac{8}{5}} p + i\eps \omega'(k) \xi }{\eps^{\frac{8}{5}} p+V(k)+ i\eps \omega'(k) \xi } \nonumber \\
=&    \frac{\eps^{\frac{8}{5}} p+V(k)  }{(\eps^{\frac{8}{5}} p+V(k))^2+(\eps \omega'(k) \xi)^2 } \eps^{\frac{8}{5}} p \nonumber  \\
& + \frac{  V  i\eps \omega'(k) \xi}{(\eps^{\frac{8}{5}} p+V(k))^2+(\eps \omega'(k) \xi)^2 } \nonumber  \\
&  +  \frac{ (\eps \omega'(k) \xi)^2 }{(\eps^{\frac{8}{5}} p+V(k))^2+(\eps \omega'(k) \xi)^2 }\label{eq:1-V}
\end{align}
Using the fact that $V(k)=V(-k)$, $\omega'(-k)=-\omega'(k)$, we deduce that
\begin{align*}
a_1^\eps(p,\xi):= & 
 -p \int_{-\frac 1 2}^{\frac 1 2}   \frac{\eps^{\frac{8}{5}} p+V(k)  }{(\eps^{\frac{8}{5}} p+V(k))^2+(\eps \omega'(k) \xi)^2 } V(k)\, dk\,  \\
&  - \eps^{-\frac{8}{5}} \int_{-\frac 1 2}^{\frac 1 2}  \frac{ (\eps \omega'(k) \xi)^2 }{(\eps^{\frac{8}{5}} p+V(k))^2+(\eps \omega'(k) \xi)^2 } V(k)\, dk .
\end{align*}
Dominated convergence immediately implies that the first term converges to $-p$, so we only have to consider the term
$$ d^\eps(p,\xi)=\eps^{-\frac{8}{5}} \int_{-\frac 1 2}^{\frac 1 2}   \frac{ (\eps \omega'(k) \xi)^2 }{(\eps^{\frac{8}{5}} p+V(k))^2+(\eps \omega'(k) \xi)^2 } V(k)\, dk .
$$
For some $\delta\in(0,\frac 1 2)$, we write
$$ d^\eps(p,\xi)= d^\eps_1(p,\xi)+d^\eps_2(p,\xi)$$
where
\begin{align*} 
d_1^\eps(p,\xi)& =  \eps^{-\frac{8}{5}} \int_{\tiny \begin{array}{l} k\in (-\frac 1 2,\frac 1 2)\\ |k|\geq \delta\end{array}}   \frac{ (\eps \omega'(k) \xi)^2 }{(\eps^{\frac{8}{5}} p+V(k))^2+(\eps \omega'(k) \xi)^2 } V(k)\, dk  \\
& \leq   C \eps^{-\frac{8}{5}} \int_{\tiny \begin{array}{l} k\in (-\frac 1 2,\frac 1 2)\\ |k|\geq \delta\end{array}}   \frac{ (\eps   \xi)^2 }{V(k)}\, dk \\ 
& \leq  C(\delta) |\xi|^2 \eps^{\frac{2}{5}}  
\end{align*}
and 
\begin{align*}
d_2^\eps(p,\xi) &  =  \eps^{-\frac{8}{5}} \int_{|k|\leq \delta}  
 \frac{ (\eps \omega'(k) \xi)^2 }{(\eps^{\frac{8}{5}} p+V(k))^2+(\eps \omega'(k) \xi)^2 } V(k)\, dk  \\
& = 2   \eps^{-\frac{8}{5}} \int_{0}^\delta  
 \frac{ (\eps\pi  \cos(\pi k) \xi)^2 }{(\eps^{\frac{8}{5}} p+W(k)|k|^{5/3})^2+(\eps\pi \cos(\pi k) \xi)^2 } W(k)|k|^{5/3}\, dk \\ 
 & = 2   \eps^{-\frac{8}{5}} \int_{0}^\delta  
 \frac{ (\pi  \cos(\pi k))^2 }{(\eps^{\frac{3}{5}}\frac{ p}{|\xi|}+W(k)\frac{|k|^{5/3}}{\eps|\xi|})^2+(\pi \cos(\pi k))^2 } W(k)|k|^{5/3}\, dk .
\end{align*}
We now do the change of variable
\begin{equation}\label{eq:cv}
 w=\frac{|k|^{5/3}}{\eps|\xi|}, \quad dk = \frac{3}{5} (\eps| \xi|)^{3/5}w^{-2/5} dw,
 \end{equation}
which yields
\begin{align*}
d_2^\eps(p,\xi) & = 2   \eps^{-\frac{8}{5}} \int_{0}^{\frac{\delta^{5/3}}{\eps|\xi|}}  
 \frac{ z^\eps(w) }{\left(\eps^{\frac{3}{5}}\frac{ p}{|\xi|}+W^\eps(w)w \right)^2+z^\eps(w) } W^\eps(w)\eps|\xi| w \frac{3}{5} (\eps |\xi|)^{3/5}w^{-2/5} \, dw \\
 & = |\xi|^{8/5}  \frac{6}{5}    \int_{0}^{\frac{\delta^{5/3}}{\eps|\xi|}}  
 \frac{ z^\eps(w) }{\left(\eps^{\frac{3}{5}}\frac{ p}{|\xi|}+W^\eps(w)w \right)^2+z^\eps(w) } W^\eps(w) w^{3/5} \, dw
\end{align*}
where
\begin{align*}
z^\eps(w) & =(\pi \cos(\pi (\eps |\xi| w)^{3/5}))^2 \\
W^\eps(w) & = W((\eps |\xi| w)^{3/5}).
\end{align*}
In particular, the integrand converges pointwise (for all $w$ and $\xi$), as $\eps$ goes to zero, to
$$ 
\frac{ \pi^2 }{\left(w_0 w \right)^2+\pi^2 } w_0 w^{3/5} $$
and it is bounded by
$$ 
\frac{ \pi^2 }{\left( C_0^{-1}w \right)^2+(\pi \cos(\pi\delta))^2} C_0 w^{3/5}.
$$
We deduce that 
$$ 
|d_2^\eps(p,\xi)| \leq C |\xi|^{8/5}$$
for some constant $C$ and that
$$ d^\eps_2(p,\xi) \longrightarrow |\xi|^{8/5}  \frac{6}{5}    \int_{0}^{\infty}   
\frac{ \pi^2 }{\left(w_0 w \right)^2+\pi^2 } w_0 w^{3/5}  \, dw = \kappa_1 |\xi|^{\frac 8 5}$$
(recall that $w_0=v_0\pi^{5/3}$) which concludes the proof of the first part.
Note that we have also proved that
$$ |a^\eps_1(p,\xi)| \leq p+C \eps^{\frac{2}{5}}  |\xi|^2 + C |\xi|^{8/5}.$$
In particular, $a^\eps_1(p,\xi)$ is bounded in $L^\infty_{loc}((0,\infty)\times\RR)$. Since it converges pointwise, a classical argument shows that it also converges strongly in $L^p_{loc}((0,\infty)\times\RR)$ for all $0<p<\infty$.

\medskip

The convergence of $a_2^\eps$  is proved similarly:
Using \eqref{eq:1-V}, we find
\begin{align*}
a_2^\eps(p,\xi):= & 
 -\eps^{\frac 3 5 }p \int_{-\frac 1 2}^{\frac 1 2}   \frac{\eps^{\frac{8}{5}} p+V(k)  }{(\eps^{\frac{8}{5}} p+V(k))^2+(\eps \omega'(k) \xi)^2 } \frac{V(k)}{\omega(k)}\, dk\,  \\
&  - \eps^{-1} \int_{-\frac 1 2}^{\frac 1 2}  \frac{ (\eps \omega'(k) \xi)^2 }{(\eps^{\frac{8}{5}} p+V(k))^2+(\eps \omega'(k) \xi)^2 } \frac{V(k)}{\omega(k)}\, dk .
\end{align*}
The first term is bounded by
\begin{align*}
& \eps^{\frac 3 5 }p \int_{-\frac 1 2}^{\frac 1 2}   \frac{ 1 }{\eps^{\frac{8}{5}} p+V(k)} \frac{V(k)}{\omega(k)}\, dk \\
& \qquad \leq C  \eps^{\frac 3 5 }p \int_{0}^{\frac 1 2}   \frac{ 1 }{\eps^{\frac{8}{5}} p+C_0^{-1} k^{\frac 5 3}} k^{\frac2 3}\, dk \\
& \qquad \leq C  \eps^{-1} \int_{0}^{\frac 1 2}   \frac{ 1 }{1+ \eps^{-\frac{8}{5}} p^{-1} k^{\frac 5 3}} k^{\frac2 3}\, dk \\
& \qquad \leq C  \eps^{-1} p \, \eps^{\frac{8}{5}}\int_{0}^{C\eps^{-\frac 8 5} p^{-1}}   \frac{ 1 }{1+ w} dw\\ 
& \qquad \leq C   p\,  \eps^{\frac{3}{5}}\ln(1+C\eps^{-\frac 8 5} p^{-1})  
\end{align*}
and thus converges to zero as $\eps\to 0$ (here we used the change of variable $w= \eps^{-\frac{8}{5}} p^{-1} k^{\frac 5 3}$).
For the second term the same decomposition of the integral in the interval $|k|\leq \delta$ and $|k|\geq \delta$. The integral in $|k|\geq \delta$ is bounded by $C(\delta) \eps |\xi|^2$. For the integral in  $|k|\leq \delta$, the change of variable \eqref{eq:cv} gives that it is bounded by $C|\xi|$ and converges to
$$  |\xi| \frac{6}{5}    \int_{0}^{\infty}   
\frac{ \pi^2 }{\left(w_0 w \right)^2+\pi^2 } \frac{w_0}{\pi}  \, dw = \kappa_2 |\xi|.$$
Note that
\begin{equation} \label{eq:estimatea2}
|a^\eps_2(p,\xi)|\leq C   p\,  \eps^{\frac{3}{5}}\ln(1+C\eps^{-\frac 8 5} p^{-1})  + C(\delta) \eps|\xi|^2 + C|\xi|
\end{equation}
so $a^\eps_2\in L^\infty_{loc}((0,\infty)\times \RR)$ implying, next to the pointwise convergence, the $L^p_{loc}((0,\infty)\times \RR)$ strong convergence for $0<p<\infty$.
\medskip

Finally,
using \eqref{eq:1-V}, we find
\begin{align}
a_3^\eps(p,\xi):= & 
 -\eps^{\frac 6 5 }p \int_{-\frac 1 2}^{\frac 1 2}   \frac{\eps^{\frac{8}{5}} p+V(k)  }{(\eps^{\frac{8}{5}} p+V(k))^2+(\eps \omega'(k) \xi)^2 } \frac{V(k)}{\omega(k)^2}\, dk\,  \nonumber \\
&  - \eps^{-\frac 2 5} \int_{-\frac 1 2}^{\frac 1 2}  \frac{ (\eps \omega'(k) \xi)^2 }{(\eps^{\frac{8}{5}} p+V(k))^2+(\eps \omega'(k) \xi)^2 } \frac{V(k)}{\omega(k)^2}\, dk .\label{eq:a3eps2}
\end{align}
The first term is bounded by
\begin{align*}
& \eps^{\frac 6 5 }p \int_{-\frac 1 2}^{\frac 1 2}   \frac{ 1 }{\eps^{\frac{8}{5}} p+V(k)} \frac{V(k)}{(\omega(k))^2}\, dk \\
& \qquad \leq C  \eps^{\frac 6 5 }p \int_{0}^{\frac 1 2}   \frac{ 1 }{\eps^{\frac{8}{5}} p+C_0^{-1} k^{\frac 5 3}} k^{-\frac1 3}\, dk \\
& \qquad \leq C  \eps^{-\frac 2 5} \int_{0}^{\frac 1 2}   \frac{ 1 }{1+ \eps^{-\frac{8}{5}} p^{-1} k^{\frac 5 3}} k^{-\frac1 3}\, dk \\
& \qquad \leq C  \eps^{\frac {24} {25} } p^{\frac 4 5} \, \int_{0}^{C\eps^{-\frac 8 5} p^{-1}}   \frac{ w^{-\frac 3 5} }{1+ w} dw\\ 
& \qquad \leq C  \eps^{\frac {24} {25} } p^{\frac 4 5}\, \int_{0}^{\infty}   \frac{ w^{-\frac 3 5} }{1+ w} dw
\end{align*}
and thus converges to zero as $\eps\to 0$.
For the second term, 
the same decomposition of the integral in the interval $|k|\leq \delta$ and $|k|\geq \delta$. The integral in $|k|\geq \delta$ is bounded by $C(\delta)\eps^{8/5}|\xi|$. The integral in $|k|\leq \delta$, the change of variable \eqref{eq:cv} gives that it is bounded by $C|\xi|^{2/5}$ and converges to
$$  |\xi|^{\frac{2}{5}} \frac{6}{5}    \int_{0}^{\infty}   
\frac{ \pi^2 }{\left(w_0 w \right)^2+\pi^2 } \frac{w_0}{\pi^2} w^{-\frac 3 5}  \, dw = \kappa_3 |\xi|^{\frac{2}{5}}.$$
Analogously as in the previous cases, we have that 
$$|a^\eps_3(p,\xi)|\leq C  \eps^{\frac {24} {25} } p^{\frac 4 5} +C(\delta) \eps^{8/5}|\xi|+ C|\xi|^{2/5}$$
so $a^\eps_3\in L^\infty_{loc}((0,\infty)\times \RR)$ which, next to the pointwise convergence, implies the $L_{loc}^p((0,\infty)\times \RR)$ strong convergence $p\in(0,\infty)$.

\end{proof}

It only remain to prove Proposition \ref{prop:remainder}. 
For that we will require the following lemma:
\begin{lemma}\label{lem:remainder}
For all $\eta\in(0,\frac 1 3]$, we have
\begin{align}
& \int_{\TT} \left| \frac{ V(k) }{\eps^{\frac{8}{5}} p+V(k)+ i\eps \omega'(k) \xi } -1\right|^2 V(k)(\sin(\pi k))^{\frac 1 3-\eta} \, dk\nonumber\\
& \qquad\qquad\qquad \qquad\qquad\qquad \leq C\left[ (\eps^\frac{8}{5} p)^{\frac 8 5} +  (\eps |\xi|)^{\frac 9 5 - \frac {3\eta}{5}}+ (\eps |\xi|)^2 \right]
\label{eq:R111}
\end{align}
and 
\begin{align}
& \int_{\TT} \left| \frac{ V(k) }{\eps^{\frac{8}{5}} p+V(k)+ i\eps \omega'(k) \xi } -1\right|^2 \frac{V(k)}{\omega(k)^2}(\sin(\pi k))^{\frac 1 3-\eta} \, dk\nonumber \\
&\qquad\qquad\qquad \qquad\qquad\qquad  \leq
C \left[(\eps^{\frac{8}{5}}p)^{\frac{3}{5}(1-\eta)} +  (\eps\xi)^{\frac{3}{5}(1-\eta)}+(\eps|\xi|)^2\right]
\label{eq:R222}
\end{align}
\end{lemma}
We note that when $\eta = \frac 1 3$ (that is when we do not have the term $(\sin(\pi k))^{\frac 1 3-\eta}$ in the integral), then the integral behaves like $\eps^{\frac 8 5}$. As we will see below, this would be just enough to show that the remainder term $R_1^\eps$ is bounded, but not to show that it converges to zero. The improvement of the norm of $K$ given by \eqref{eq:bdK} is thus essential here.

We first prove Proposition \ref{prop:remainder} (using Lemma \ref{lem:remainder}), before giving the proof of Lemma \ref{lem:remainder}:

\begin{proof}[Proof of Proposition \ref{prop:remainder}]
Using \eqref{eq:bdK}, we get:
\begin{align*}
\left| R_1^\eps(p,\xi) \right| &=  \eps^{-\frac{8}{5}} \left| \int_{\TT} \left( \frac{ V(k) }{\eps^{\frac{8}{5}} p+V(k)+ i\eps \omega'(k) \xi } -1\right) K(\hfe-\Pi(\hfe))(k)\, dk \right|\\
& \leq   \eps^{-\frac{8}{5}} \left(\int_{\TT} \left| \frac{ V(k) }{\eps^{\frac{8}{5}} p+V(k)+ i\eps \omega'(k) \xi } -1\right|^2 V(k)(\sin(\pi k))^{\frac 1 3-\eta}\, dk\right)^{1/2} \\
&\qquad \qquad \times \left( \int_\TT K(\hfe-\Pi(\hfe))^2 (\sin(\pi k))^{-\frac 1 3+\eta} V^{-1}(k)\, dk\right)^{1/2}\\
& \leq   C\eps^{-\frac{8}{5}} \left(\int_{\TT} \left| \frac{ V(k) }{\eps^{\frac{8}{5}} p+V(k)+ i\eps \omega'(k) \xi } -1\right|^2 V(k)(\sin(\pi k))^{\frac 1 3-\eta} \, dk\right)^{1/2} \\
& \qquad\qquad \times \left( \int_\TT (\hfe-\Pi(\hfe))^2 (\sin(\pi k))^{\frac 1 3-\eta} V(k)\, dk\right)^{1/2}\\
\end{align*}
and using \eqref{eq:R111}, we deduce that for $p<P$ and $ |\xi|\leq K$, we have
\begin{align*}
\left| R_1^\eps(p,\xi) \right|  \leq  C(P,K) \eps^{-\frac{8}{5}} \eps^{\frac {9-3\eta} {10} } \left( \int_\TT (\hfe-\Pi(\hfe))^2 V(k)\, dk\right)^{1/2}.
\end{align*}
For all $0<a<P$ and $K>0$, we deduce
\begin{align*}
\int_a^P\int_{-K}^K \left| R_1^\eps(p,\xi) \right|^2 \, d\xi \, dp & \leq C(P,K) \eps^{-\frac{16}{5}} \eps^{\frac {9-3\eta} {5} } \|  \hfe-\Pi(\hfe)\|^2_{L^\infty((a,\infty);L^2_V(\TT\times\RR))}\\
& \leq C(P,K)\frac{1}{ 2a}   \eps^{\frac {-7-3\eta} {5} }  \| f^\eps-\Pi(f^\eps)\|^2_{L^2((0,\infty);L^2_V(\TT\times\RR))} \\
& \leq C(a,P,K)   \eps^{\frac {-7-3\eta} {5} }  \eps^{\frac 8 5}\\
& \leq C(a,P,K)   \eps^{\frac {1-3\eta} {5} } 
\end{align*}
where we have used  \eqref{eq:laplaceL^2}.

Clearly, this implies Proposition \ref{prop:remainder} for $i=1$.

\medskip

Proceeding similarly, we have that
\begin{align}
|R^\eps_2(p, \xi)| & \leq C\eps^{-1} \lp  (\eps^{\frac{8}{5}}p)^{\frac{3}{5}(1-\eta)} +  (\eps\xi)^{\frac{3}{5}(1-\eta)}+(\eps|\xi|)^2\rp^{1/2} \nonumber\\
& \qquad\qquad\qquad\qquad   \times \lp \int_\TT (\widehat f^\eps - \Pi (\widehat f^\eps) )^2 V(k)dk\rp^{1/2}
 \label{eq:estimate_R2}
\end{align}
and  therefore
\begin{align} \nonumber
\int_a^P\int_{-K}^K \left| R_2^\eps(p,\xi) \right| ^2  \, dp\,d\xi & \leq C(P,K) \eps^{-2}\eps^{\frac{3(1-\eta)}{5}}  \|  (\hfe-\Pi(\hfe))\|^2_{L^\infty((a,\infty);L^2_V(\TT\times\RR))}\\ \nonumber
& \leq C(a,P,K)   \eps^{\frac {-7-3\eta} {5} }  \| f^\eps-\Pi(f^\eps)\|^2_{L^2((0,\infty);L^2_V(\TT\times\RR))} \\ \label{eq:estimateR2}
& \leq C(a,P,K)   \eps^{\frac {1-3\eta} {5} } 
\end{align}
which converges to zero for any $\eta\in(0,\frac{1}{3})$.
\end{proof}

\begin{proof}[Proof of Lemma \ref{lem:remainder}]
We write:
\begin{align*}
& \int_{\TT} \left| \frac{ V(k) }{\eps^{\frac{8}{5}} p+V(k)+ i\eps \omega'(k) \xi } -1\right|^2 V(k)(\sin(\pi k))^{\frac 1 3-\eta} \, dk \\
& \qquad = \int_{\TT} \left| \frac{ \eps^{\frac{8}{5}} p+ i\eps \omega'(k) \xi  }{\eps^{\frac{8}{5}} p+V(k)+ i\eps \omega'(k) \xi }\right|^2 V(k)(\sin(\pi k))^{\frac 1 3-\eta} \, dk \\
& \qquad = \int_{\TT} \frac{ (\eps^{\frac{8}{5}} p)^2+ (\eps \omega'(k) \xi)^2  }{(\eps^{\frac{8}{5}} p+V(k))^2+ (\eps \omega'(k) \xi)^2 } V(k)(\sin(\pi k))^{\frac 1 3-\eta} \, dk \\
& \qquad = I_1 + I_2 
\end{align*}
where
\begin{align*}
I_1 & := \int_{\TT} \frac{ (\eps^{\frac{8}{5}} p)^2 }{(\eps^{\frac{8}{5}} p+V(k))^2+ (\eps \omega'(k) \xi)^2 } V(k)(\sin(\pi k))^{\frac 1 3-\eta} \, dk \\
& \leq2 \int_{0}^{1/2} \frac{ (\eps^{\frac{8}{5}} p)^2 }{(\eps^{\frac{8}{5}} p+V(k))^2 } V(k)(\sin(\pi k))^{\frac 1 3-\eta} \, dk\\
& \leq2 \int_{0}^{1/2} \frac{ (\eps^{\frac{8}{5}} p)^2 }{(\eps^{\frac{8}{5}} p+k^{5/3})^2 } k^{5/3} \, dk
\end{align*}
(note we do not need to use the $(\sin(\pi k))^{\frac 1 3-\eta} $ to control this term) and 
\begin{align*}
I_2 & := \int_{\TT} \frac{ (\eps \omega'(k) \xi)^2  }{(\eps^{\frac{8}{5}} p+V(k))^2+ (\eps \omega'(k) \xi)^2 } V(k)(\sin(\pi k))^{\frac 1 3-\eta} \, dk \\
& \leq  2\int_{0}^{1/4} \frac{ (\eps \omega'(k) \xi)^2  }{V(k)^2+ (\eps \omega'(k) \xi)^2 } V(k)\,(\sin(\pi k))^{\frac 1 3-\eta}\, dk + 2 \int_{1/4}^{1/2} \frac{ (\eps\pi \xi)^2  }{V(k) }  \, dk\\
& \leq  C\int_{0}^{1/4} \frac{ (\eps  \xi)^2  }{k^{10/3}+ (\eps   \xi)^2 } k^{2-\eta} \, dk + C\eps^2 |\xi|^2
\end{align*}
(here the $(\sin(\pi k))^{\frac 1 3-\eta} $ is essential).

Using the change of variable $w= \frac{k^{5/3}}{\eps^{8/5}p}$ in $I_1$ and $w=\frac{k^{5/3}}{\eps \xi} $ in $I_2$, we find
\begin{align*}
I_1 & \leq C (\eps^{8/5} p )^{8/5} \int_0^\infty \frac{w^{3/5}}{(1+w)^2}\, dw\\
I_2 & \leq C (\eps\xi)^{9/5-3\eta/5}\int_0^{\infty}\frac{w^{4/5-3\eta/5}}{1+w^2}\, dw+ C\eps^2 |\xi|^2\\
\end{align*}
where the integral in the right hand side are clearly finite (recall that $\eta \in (0,\frac 1 3)$.
Inequality \eqref{eq:R111} follows.
\medskip

We now proceed similarly to prove \eqref{eq:R222}: 
First, we write
\begin{align*}
& \int_{\TT} \left| \frac{ V(k) }{\eps^{\frac{8}{5}} p+V(k)+ i\eps \omega'(k) \xi } -1\right|^2 \frac{V(k)}{\omega(k)^2}(\sin(\pi k))^{\frac 1 3-\eta} \, dk\\
& \qquad = \int_{\TT} \frac{ (\eps^{\frac{8}{5}} p)^2+ (\eps \omega'(k) \xi)^2  }{(\eps^{\frac{8}{5}} p+V(k))^2+ (\eps \omega'(k) \xi)^2 } \frac{V(k)}{\omega(k)^2} (\sin(\pi k))^{\frac 1 3-\eta} \, dk \\
& \qquad = \tilde I_1 + \tilde I_2 
\end{align*}
where
\begin{align*}
\tilde I_1 & := \int_{\TT} \frac{ (\eps^{\frac{8}{5}} p)^2 }{(\eps^{\frac{8}{5}} p+V(k))^2+ (\eps \omega'(k) \xi)^2 } \frac{V(k)}{\omega(k)^2} (\sin(\pi k))^{\frac 1 3-\eta} \, dk \\
& \leq2 \int_{0}^{1/2} \frac{ (\eps^{\frac{8}{5}} p)^2 }{(\eps^{\frac{8}{5}} p+V(k))^2 } \frac{V(k)}{\omega(k)^2}\,(\sin(\pi k))^{\frac 1 3-\eta}  dk\\
& \leq2 \int_{0}^{1/2} \frac{ (\eps^{\frac{8}{5}} p)^2 }{(\eps^{\frac{8}{5}} p+k^{5/3})^2 }  k^{-\eta} \, dk
\end{align*}
and 
\begin{align*}
\tilde I_2 & := \int_{\TT} \frac{ (\eps \omega'(k) \xi)^2  }{(\eps^{\frac{8}{5}} p+V(k))^2+ (\eps \omega'(k) \xi)^2 } \frac{V(k)}{\omega(k)^2} (\sin(\pi k))^{\frac 1 3-\eta} \, dk \\
& \leq  2\int_{0}^{1/2} \frac{ (\eps \omega'(k) \xi)^2  }{V(k)^2+ (\eps \omega'(k) \xi)^2 } \frac{V(k)}{\omega(k)^2} (\sin(\pi k))^{\frac 1 3-\eta} \, dk\\
& \leq  2\int_{0}^{1/4} \frac{ (\eps  \xi)^2  }{k^{10/3}+ (\eps   \xi)^2 } k^{-\eta} \, dk + C(\eps|\xi|)^2
\end{align*}
Using the change of variable $w= \frac{k^{5/3}}{\eps^{8/5}p}$ in $\tilde I_1$ and $w=\frac{k^{5/3}}{\eps \xi} $ in $\tilde I_2$, we find
\begin{align*}
\tilde I_1 & \leq C (\eps^{\frac{8}{5}}p)^{\frac{3}{5}(1-\eta)}  \int_0^{\infty}\frac{w^{-3/5\eta}}{(1+w)^2}\, dw  \\
\tilde I_2 & \leq C (\eps\xi)^{\frac{3}{5}(1-\eta)}\int_0^{\infty}\frac{w^{-2/5-3\eta/5}}{1+w^2}\, dw + C(\eps|\xi|)^2\\
\end{align*}
which yields \eqref{eq:R222}.
\end{proof}

\section{Proof of Proposition \ref{prop:convS}}\label{sec:Seps}
The proof of Proposition \ref{prop:convS} relies 
on the following crucial bound:
\begin{lemma}\label{lem:a_3}
There exists a constant $c$ such that for all $K$ and for all  $\eps$ such that $\eps K\leq 1$, 
the following lower bound holds
\begin{equation} \label{eq:lowerestimatefora3}
|a^\eps_3(p,\xi)|\geq  c \, \eps^{\frac{6}{25}} p^\frac{2}{5} + c |\xi|^{\frac{2}{5}} \qquad\mbox{ for } 0\leq p\leq K, \quad |\xi| \leq K.
\end{equation}
\end{lemma}

\begin{proof}[Proof of Lemma \ref{lem:a_3}]
We recall that $a^\eps_3(p,\xi)$ is given by \eqref{eq:a3eps2}.
In particular, we note that for all $(p,\xi)\neq (0,0)$, we have  $a^\eps_3(p,\xi)<0$.
Furthermore, we can write (using the fact that all the terms in \eqref{eq:a3eps2} have the same sign):
\begin{align*} 
-a^\eps_3(p,k)  & \geq  \eps^{\frac 6 5 }p \int_{0}^{\frac 1 4}   \frac{\eps^{\frac{8}{5}} p+V(k)  }{(\eps^{\frac{8}{5}} p+V(k))^2+(\eps \omega'(k) \xi)^2 } \frac{V(k)}{\omega(k)^2}\, dk\\
& \quad  + \eps^{-\frac 2 5} \int_{0}^{\frac 1 4}  \frac{ (\eps \omega'(k) \xi)^2 }{(\eps^{\frac{8}{5}} p+V(k))^2+(\eps \omega'(k) \xi)^2 } \frac{V(k)}{\omega(k)^2}\, dk.
\end{align*}
Using the fact that for $k\in(0,1/4)$ we have $C_0^{-1} |k|^{5/3} \leq V(k) \leq C_0 |k|^{5/3}$, $\frac{\pi}{\sqrt 2}\leq \omega'(k)\leq \pi$ and $ \frac{\pi}{2} k \leq \omega(k)\leq \pi k$, we obtain the following lower bound (for some constant $c>0$):
\begin{align} 
-a^\eps_3(p,k)  & \geq c \,\eps^{\frac 6 5 }p \int_{0}^{\frac 1 4}   \frac{ |k|^{4/3}  }{(\eps^{\frac{8}{5}} p+C_0 |k|^{5/3})^2+( \eps\pi  \xi)^2 } \, dk\nonumber\\
& \quad  +  c\, \eps^{-\frac 2 5}  \int_{0}^{\frac 1 4}  \frac{ (\eps \pi \xi)^2 k^{-1/3} }{(\eps^{\frac{8}{5}} p+C_0 |k|^{5/3})^2+(\eps\pi \xi)^2 } \, dk. \label{eq:a3eps3}
\end{align}

From now on, we fix $K$ and assume that $0<p\leq K$ and that $|\xi|\leq K$.
We also assume that $\eps$ is such that $\eps K \leq 1$.
In order to establish \eqref{eq:lowerestimatefora3}, we consider two cases, and in each case we use only one of the integrals in \eqref{eq:a3eps3}:
\begin{enumerate}
\item First, assume that $p$ and $\xi$ are such that
\begin{equation}\label{eq:xip} 
|\xi|\leq \eps^{\frac 3 5} p.
\end{equation}
Then, using only the first integral in \eqref{eq:a3eps3}, we get (using \eqref{eq:xip}):
\begin{align*} 
-a^\eps_3(p,k) 
 & \geq  c\, \eps^{\frac 6 5 }p  \int_{0}^{\frac 1 4}   \frac{ |k|^{4/3}  }{(\eps^{\frac{8}{5}} p+C_0 |k|^{5/3})^2+( \pi \eps^{\frac 8 5} p  )^2 } \, dk
\end{align*}
and the change of variable $w=(\eps^{\frac 8 5}p)^{-\frac{3}{5}} k$ yields
$$
-a^\eps_3(p,k) 
\geq c\, \eps^{\frac 6 5 }p \frac{(\eps^{\frac{8}{5}} p)^{\frac 7 5}}{(\eps^{\frac{8}{5}} p)^2}
 \int_{0}^{\frac 1 {4(\eps^{\frac{8}{5}} p)^{\frac{3}{5}}}}   \frac{ |w|^{4/3}  }{(1+C_0 |w|^{5/3})^2+\pi ^2 } \, dw
$$
and using the fact that $ \eps^{\frac{8}{5}} p \leq 1$, we deduce (for a different constant $c$):
$$-a^\eps_3(p,k)  \geq c \, \eps^{\frac{6}{25}} p^\frac{2}{5}.$$
Finally, using \eqref{eq:xip}, we also  get
$$-a^\eps_3(p,k)  \geq  c |\xi|^{\frac 2 5}$$
and so \eqref{eq:lowerestimatefora3} holds in this case.

\item Next, we assume that
$p$ and $\xi$ are such that
\begin{equation}\label{eq:pxi} 
 \eps^{\frac 3 5} p\leq |\xi|
\end{equation}
and using  only the second integral in \eqref{eq:a3eps3}, we get (using \eqref{eq:pxi}):
\begin{align*} 
-a^\eps_3(p,k)  & \geq c\,  \eps^{-\frac 2 5}  \int_{0}^{\frac 1 4}  \frac{ (\eps \pi \xi)^2 k^{-1/3} }{(\eps |\xi|+C_0 |k|^{5/3})^2+(\eps\pi \xi)^2 } \, dk 
\end{align*}
and the change of variable $w=(\eps|\xi|)^{-\frac 3 5}k$, yields:
\begin{align*} 
-a^\eps_3(p,k)  & \geq c\,\eps^{-\frac 2 5} \pi^2 \int_{0}^{\frac 1 {4(\eps \xi)^{\frac 3 5 }}}  \frac{ (\eps  \xi)^\frac{2}{5} w^{-1/3} }{(1+C_0 |w|^{5/3})^2+\pi^2 } \, dw \\
& \geq 
c |\xi|^{\frac 2 5} 
\end{align*}
(using the fact that $\eps|\xi|\leq 1$).
Finally, using \eqref{eq:pxi}, we also  get
$$-a^\eps_3(p,k)  \geq  c \, \eps^{\frac{6}{25}} p^\frac{2}{5}$$
and so \eqref{eq:lowerestimatefora3} holds also in this case.
\end{enumerate}
\end{proof}

\begin{proof}[Proof of Proposition \ref{prop:convS}]
We use equation \eqref{eq:asymptotic2} to determine $\widehat S^\eps$:
\begin{equation}\label{eq:SSSS}
\widehat S^\eps= \frac{1}{a^\eps_3}\left(-\mathcal{F}_2^\eps(\widehat f_0)-R^\eps_2 -a^\eps_2 \widehat T^\eps\right).
\end{equation}
Note that we can do this since $a^\eps_3(p,\xi)<0$ as long as $p$ and $\xi$ are not simultaneously zero.

We now need to show that we can pass to the limit in all the terms in the right hand side.
First, using Lemma \ref{lem:a_3}
and the  estimate \eqref{eq:estimatea2}, we  deduce that for a given  $K$ and for all $\eps \leq K^{-1}$, we have
\begin{align*}
\frac{a^\eps_2(p,\xi)}{a^\eps_3(p,\xi)}=\left| \frac{a^\eps_2(p,\xi)}{a^\eps_3(p,\xi)}\right| &\leq  
\frac{Cp\eps^{3/5}\ln(1+C\eps^{-8/5}p^{-1})}{c\eps^{\frac{6}{25}} p^{\frac 2 5}} + \frac{C\eps  |\xi|^2}{c|\xi|^\frac{2}{5}} +\frac{C |\xi|}{c|\xi|^{\frac 2 5}}\\
& \leq  
Cp^{\frac 3 5}\eps^{\frac{9}{25}} \ln(1+C\eps^{-8/5}p^{-1})+ C\eps  |\xi|^\frac{8}{5} + C |\xi|^\frac{3}{5} \\
& \leq C(K)
\end{align*}
for all $0\leq p\leq K$ and $|\xi|\leq K$.
Furthermore, this uniform bound, together with Proposition \ref{prop:aeps1} implies that
$$\frac{a^\eps_2(p,\xi)}{a^\eps_3(p,\xi)} \longrightarrow \frac{\kappa_2}{\kappa_3}|\xi|^{3/5}$$
point-wise and in $L^p_{loc}((0,\infty)\times \RR)$ strong.

\medskip

Next, for $\eps$ sufficiently small we can use Lemma \ref{lem:a_3} along with the estimates on $\mathcal{F}^\eps_2(\hat{f}_0)$ in \eqref{eq:estimateF2} to conclude that
$$\frac{1}{a^\eps_3}\mathcal{F}^\eps_2(\hat{f}_0) \rightarrow 0 \qquad \mbox{in } \mathcal{D}'((0,\infty)\times \RR).$$

Finally, we need to bound the quantity
$$\left|\frac{R^\eps_2(p,\xi)}{a^\eps_3(p,\xi)}\right|. $$
For that, we fix $0<a<P$ and for $p\in(a,P)$ and $|\xi|\leq K$,
estimate  \eqref{eq:estimate_R2} then implies
\begin{align*}
|R^\eps_2(p, \xi)| & \leq C\eps^{-1} \lp  \eps^{\frac{24}{25} (1-\eta)} +  (\eps\xi)^{\frac{3}{5}(1-\eta)}\rp^{1/2}    \lp \int_\TT (\widehat f^\eps - \Pi (\widehat f^\eps) )^2 V(k)dk\rp^{1/2}\\
& \leq \left\{
\begin{array}{ll}
C\eps^{-1}   \eps^{\frac{12}{25} (1-\eta)} \lp \int_\TT (\widehat f^\eps - \Pi (\widehat f^\eps) )^2 V(k)dk\rp^{1/2} & \mbox{ if }  |\xi|\leq \eps^{3/5} \\
C\eps^{-1}  (\eps\xi)^{\frac{3}{10}(1-\eta)} \lp \int_\TT (\widehat f^\eps - \Pi (\widehat f^\eps) )^2 V(k)dk\rp^{1/2} & \mbox{ if }  |\xi|\geq \eps^{3/5} \\
\end{array}
\right.
\end{align*}
and we are going to use the following consequence of Lemma \ref{lem:a_3}:
$$ -a^\eps_3(p,\xi) \geq \left\{ \begin{array}{ll}
c(a)\eps^{6/25} & \mbox { if } |\xi|\leq \eps^{3/5}\\
c|\xi|^{2/5}  & \mbox { if } |\xi|\geq \eps^{3/5}
\end{array}
\right.
$$ 
We deduce
\begin{eqnarray*}
\left|\frac{R^\eps_2(p,\xi)}{a^\eps_3(p,\xi)}\right| &\leq& 
\left\{\begin{array}{ll}
C  \eps^{\frac{-19-12\eta}{25}} \lp \int_\TT (\widehat f^\eps - \Pi (\widehat f^\eps) )^2 V(k)dk\rp^{1/2} & \mbox{ if }  |\xi|\leq \eps^{3/5} \\
C\eps^{-1}  \eps^{\frac{3}{10}(1-\eta)} |\xi|^{\frac{-1-3\eta}{10}}\lp \int_\TT (\widehat f^\eps - \Pi (\widehat f^\eps) )^2 V(k)dk\rp^{1/2} & \mbox{ if }  |\xi|\geq \eps^{3/5} \\
\end{array}
\right. .
\end{eqnarray*}
Finally, using the condition $ |\xi|\geq \eps^{3/5}$ in the second case, we deduce that
$$
\left|\frac{R^\eps_2(p,\xi)}{a^\eps_3(p,\xi)}\right| \leq 
C(a,p,K)  \eps^{\frac{-19-12\eta}{25}} \lp \int_\TT (\widehat f^\eps - \Pi (\widehat f^\eps) )^2 V(k)dk\rp^{1/2} 
$$
for all  $p\in(a,P)$ and $|\xi|\leq K$.

We deduce
\begin{align*} \nonumber
\left( \int_a^P\int_{-K}^K
\left| \frac{R_2^\eps(p,\xi)}{a_3^\eps(p,\xi)} \right|^2 \, d\xi \, dp\right)^{1/2} & \leq C(a,P,K)  \eps^{\frac{-19-12\eta}{25}} \|  (\hfe-\Pi(\hfe))\|_{L^\infty((a,\infty);L^2_V(\TT\times\RR))}\\ \nonumber
& \leq C(a,P,K)   \eps^{\frac {-19-12\eta} {25} }  \| f^\eps-\Pi(f^\eps)\|_{L^2((0,\infty);L^2_V(\TT\times\RR))}\\
& \leq C(a,P,K)   \eps^{\frac {1-12\eta} {25} } 
\end{align*}
which goes to zero as $\eps\to0$.
\medskip

We can now pass to the limit in \eqref{eq:SSSS} to conclude that 
$$\widehat S^\eps \longrightarrow \widehat S =\frac{\kappa_2}{\kappa_3}|\xi|^{3/5}\widehat T\qquad \mbox{in } \mathcal{D}'((0,\infty)\times \RR)$$
which completes the proof of Proposition \ref{prop:convS}.
\end{proof}

\bibliographystyle{plain}
\bibliography{biblio_FPU}

\signam
\signsm

 \end{document}